\documentclass[seceqn,MSNbibl,number,citesort,dvips]{arxbj}

\aid{0}
\volume{20}
\issue{3}
\pubyear{2014}
\firstpage{1404}
\lastpage{1431}
\doi{10.3150/13-BEJ527} 

\makeatletter
\newcommand{\eqref}[1]{(\ref{#1})}
\newcommand{\ignore}[1]{}{}

\def\E{{\mathbb{E}}}
\def\Var{\mathrm{Var}}
\def\cov{\mathrm{Cov}}
\def\P{{\mathbb{P}}}

\newtheorem{theorem}{Theorem}[section]
\newtheorem{lem}[theorem]{Lemma}
\newproclaim{deff}[theorem]{Definition}
\newtheorem{cor}[theorem]{Corollary}
\newproclaim{prop}[theorem]{Proposition}
\newremark{rem}[theorem]{Remark}

\newcommand{\Binomial}{\operatorname{Binomial}}
\def\sfrac#1#2{#1/#2}

\def\afrac#1#2{#1/(#2)}

\def\sklfrac#1#2{(#1/#2)}
\makeatother

\begin{document}
\begin{frontmatter}

\title{Discretized normal approximation by Stein's method}
\runtitle{Discretized normal approximation}

\begin{aug}
\author{\inits{X.}\fnms{Xiao} \snm{Fang}\ead[label=e1]{stafx@nus.edu.sg}}
\runauthor{X. Fang} 
\address{Department of Statistics and Applied Probability,
National University of Singapore,
6 Science Drive 2,
Singapore 117546,
Republic of Singapore. \printead{e1}}
\end{aug}

\received{\smonth{1} \syear{2012}}
\revised{\smonth{2} \syear{2013}}

%
\begin{abstract}
We prove a general theorem to bound the total variation distance
between the distribution
of an integer valued random variable of interest and an appropriate
discretized normal distribution.
We apply the theorem to $2$-runs in a sequence of i.i.d. Bernoulli
random variables, the number of
vertices with a given degree in the Erd\"os--R\'enyi random graph, and
the uniform multinomial occupancy model.
\end{abstract}

%

\end{frontmatter}
%
\section{Introduction and the main result}
Let $S$ be a sum of independent random variables.
The Berry-Esseen theorem gives a bound on the Kolmogorov distance
between the distribution of $S$ and the normal distribution with the
same mean and variance as $S$.
%
\begin{theorem}[(Berry \cite{Be41}, Esseen \cite{Es42})]\label{thm1}
Assume $S=\sum_{i=1}^n X_i$ where $\{X_1,\dots, X_n\}$ are
independent random variables with $\E X_i=\mu_i$, $\Var X_i=\sigma
_i^2$, $\E|X_i-\mu_i|^3=\gamma_i$. Let $\mu=\sum_{i=1}^n \mu_i$,
$\sigma^2=\sum_{i=1}^n \sigma_i^2, \gamma=\sum_{i=1}^n \gamma_i$. Then,
%
\begin{equation}
\label{1} d_K\bigl(\mathcal{L(S)}, N\bigl(\mu,
\sigma^2\bigr)\bigr) \le c \gamma/\sigma^3,
\end{equation}
where $c$ is an absolute constant and
\begin{eqnarray*}
d_K\bigl(\mathcal{L}(X), \mathcal{L}(Y)\bigr)=\sup
_{z\in\mathbb{R}} \bigl| \P(X\le z)-\P(Y\le z) \bigr|.
\end{eqnarray*}
\end{theorem}
From \eqref{1}, if $\sigma^{-2}=\mathrm{O}(1/n)$ and $\gamma=\mathrm{o}(n^{3/2})$, then
%
\begin{equation}
\label{1.1} d_K\bigl(\mathcal{L(S)}, N\bigl(\mu,
\sigma^2\bigr)\bigr) \rightarrow0 \qquad \mbox{as } n\rightarrow\infty.
\end{equation}
A stronger distance, the total variation between two distributions, is
defined as
%
\begin{equation}
\label{1.2} d_{\mathrm{TV}}\bigl(\mathcal{L}(X), \mathcal{L}(Y)\bigr)=\sup
_{A\subset\mathbb{R}} \bigl|\P(X\in A)-\P(Y\in A)\bigr|.
\end{equation}
If $S$ is integer valued, the convergence in (\ref{1.1}) is no longer
valid under total variation distance because
%
\begin{equation}
\label{2} d_{\mathrm{TV}}\bigl(\mathcal{L}(S), N\bigl(\mu,
\sigma^2\bigr)\bigr)=1 \qquad \forall n\ge1.
\end{equation}
Equation (\ref{2}) follows by taking $A$ to be the set of integers in
the definition of total variation distance. Therefore, we need to find
limiting distributions other than $N(\mu,\sigma^2)$ if small total
variation distance is desired.
Several alternatives have been studied, e.g., translated Poisson
distribution \cite{Ro05,Ro07}, shifted binomial
distribution \cite{Ro08} and a new family of discrete distributions
\cite{GoXi06}.
A more natural limiting distribution, discretized normal distribution
$N^d(\mu,\sigma^2)$, is defined to be supported on the integer set
$\mathbb{Z}$ and have probability mass function at any integer $z\in
\mathbb{Z}$ as
%
\begin{equation}
\label{6.1-0} \P\bigl(z-\tfrac{1}{2} \le Z_{\mu,\sigma^2} <z+
\tfrac{1}{2}\bigr),
\end{equation}
where $Z_{\mu,\sigma^2}$ is a Gaussian variable with mean $\mu$ and
variance $\sigma^2$.

Using Stein's method, Chen and Leong \cite{ChLe10} (see also Theorem~7.4 of \cite{ChGoSh11}) proved a bound on $d_{\mathrm{TV}}(\mathcal{L}(S),
N^d(\mu,\sigma^2))$ for sums of independent integer valued random
variables. Stein's method was introduced by Stein \cite{St72}, and has
become an important approach in proving distributional approximations
because of its power in handling dependence within random variables. We
refer to \cite{BaCh05} for an introduction to Stein's method.

Chen and Leong \cite{ChLe10} used the zero-bias coupling approach in
Stein's method to obtain their result. In this paper, we develop a
different approach in Stein's method for discretized normal
approximation. Our approach not only recovers the result of Chen and
Leong \cite{ChLe10}, but also works for general integer valued random
variables. We work under the framework of Stein coupling, a concept
introduced by Chen and R\"ollin \cite{ChRo10} under which normal
approximation results can be proved.
%
\begin{deff}\label{6-d1}
Let $S$ be a random variable with mean $\mu$. We say a triple of
square-integrable random variables $(S,S',G)$ is a Stein coupling if
%
\begin{equation}
\label{6-d1-1} \E\bigl\{G f\bigl(S'\bigr)-G f(S)\bigr\} =\E(S-\mu)
f(S)
\end{equation}
for all $f$ such that the above expectations exist.
\end{deff}
The above definition is adapted from \cite{ChRo10} and includes many
of the coupling structures employed in Stein's method such as local
dependence, exchangeable pairs, and size biasing. These coupling
structures are discussed in Section~\ref{sec2}. Under the framework of Stein
coupling, we obtain the following theorem.
%
\begin{theorem}\label{6-t1}
Let $S$ be an integer valued random variable with mean $\mu$ and
finite variance $\sigma^2$. Suppose we can construct a Stein coupling
$(S,S',G)$. Then, with $D=S'-S$,
%
\begin{eqnarray}
\label{6-t1-1} &&d_{\mathrm{TV}}\bigl(\mathcal{L}(S), N^d\bigl(\mu,
\sigma^2\bigr)\bigr)\nonumber
\\
&&\quad \le\frac{2}{\sigma^2}\sqrt{\Var\bigl(\E(GD|S)\bigr)}+\sqrt{\frac{\pi
}{8}}
\frac{\E|GD^2|}{\sigma^3} +\frac{\sqrt{\E G^2 D^4}}{\sigma
^3}
\\
&&\qquad {} + \frac{1}{2\sigma^2}\E \bigl[ \bigl(\bigl|GD^2\bigr|+ |GD|\bigr)
d_{\mathrm{TV}} \bigl(\mathcal {L}(S| \mathcal{F}), \mathcal{L}(S+1|\mathcal{F})
\bigr) \bigr],\nonumber
\end{eqnarray}
where $\mathcal{F}$ is a $\sigma$-field such that $\sigma
(G,D)\subset\mathcal{F}$ where $\sigma(\cdot)$ denotes the $\sigma
$-field generated by a random variable.
\end{theorem}
%
\begin{rem}
The discretization defined in \eqref{6.1-0} has no loss of generality.
For example, one may define another discretized normal distribution
$\tilde{N}^d(\mu,\sigma^2)$ with probability mass function at $z$ as
\begin{eqnarray*}
\P(z\le Z_{\mu,\sigma^2}<z+1).
\end{eqnarray*}
Then,
\begin{eqnarray*}
d_{\mathrm{TV}}\bigl(N^d\bigl(\mu,\sigma^2\bigr),
\tilde{N}^d\bigl(\mu,\sigma^2\bigr)\bigr)&=&
d_{\mathrm{TV}}\bigl(N^d\bigl(\mu,\sigma^2\bigr),
N^d\bigl(\mu-\tfrac{1}{2},\sigma^2\bigr)\bigr)
\\
&\le& d_{\mathrm{TV}}\bigl(N\bigl(\mu,\sigma^2\bigr), N\bigl(\mu-
\tfrac{1}{2},\sigma^2\bigr)\bigr)
\\
&\le& c/\sigma,
\end{eqnarray*}
where $c$ is an absolute constant.
It can be seen from \eqref{1101} in the proof of Theorem~\ref{6-t1} that
the bound \eqref{6-t1-1} will only differ by a constant factor if one
changes the limiting distribution from $N^d(\mu,\sigma^2)$ to $\tilde
{N}^d(\mu,\sigma^2)$.
\end{rem}
%
\begin{rem}
The first three terms in the bound (\ref{6-t1-1}) are comparable to
those appearing in the upper bounds of the Kolmogorov or Wasserstein
distance for normal approximations (see, e.g., Corollary~2.2 of \cite{ChRo10}).
The last term in the bound (\ref{6-t1-1}) arises because we are
working in the total variation distance. It is easy to see that such a
term must appear by considering the case when $S$ has support
restricted to the even integers.
Also in bounding this term, we choose appropriate $\mathcal{F}$ so
that $d_{\mathrm{TV}} (\mathcal{L}(S|\mathcal{F}), \mathcal{L}(S+1|\mathcal
{F}))$ is relatively easy to bound, yet of the same order as $d_{\mathrm{TV}}
(\mathcal{L}(S|G, D), \mathcal{L}(S+1|G, D))$.
\end{rem}

R\"ollin and Ross \cite{RoRo12} provided a general method of bounding
$d_{\mathrm{TV}}(\mathcal{L}(V),\mathcal{L}(V+1))$ for a given integer valued
random variable $V$. It is our main tool for bounding the last term in
the bound \eqref{6-t1-1}.
%
\begin{lem}[(R\"ollin and Ross \cite{RoRo12})] \label{6-l1}
For a given integer valued random
variable $V$, if we can construct an exchangeable pair $(V,V')$ (i.e.,
$\mathcal{L}(V,V')=\mathcal{L}(V',V)$) so that $P(V-V'=1)\ne0$, then
%
\begin{eqnarray}
\label{6-l1-1} &&d_{\mathrm{TV}}\bigl(\mathcal{L}(V),\mathcal{L}(V+1)\bigr)\nonumber
\\[-8pt]\\[-8pt]
&&\quad \le\frac{\sqrt{\Var(\E(I(V-V'=1)|V))} + \sqrt{\Var(\E
(I(V-V'=-1)|V))}}{P(V-V'=1)}.\nonumber
\end{eqnarray}
\end{lem}
%
\begin{rem}\label{exchconst}
To apply Lemma~\ref{6-l1}, we need to construct exchangeable pairs
such that the bound in \eqref{6-l1-1} is small. A useful method to
construct such exchangeable pairs when $V$ is a function of independent
random variables is as follows. Suppose $V=f(X_1,\dots, X_n)$ where $\{
X_1,\dots, X_n\}$ are independent. Let $I$ be an independent uniform
random index from $\{1,\ldots,n\}$. Given $I$, let $X_I'$ be an
independent copy of $X_I$. Define $V'=f(X_1,\dots, X_I'\dots, X_n)$.
Then $(V,V')$ is an exchangeable pair. We will use this construction in
all the applications considered in this paper.\vadjust{\goodbreak}
\end{rem}

The remaining of the paper is organized as follows. In Section~\ref{sec2}, we
show the utility of Theorem~\ref{6-t1} by adapting it to local
dependence, exchangeable pairs, and size biasing, and bounding the
total variation distance for discretized normal approximations for
$2$-runs in a sequence of i.i.d. Bernoulli random variables, the number
of vertices with a given degree in the Erd\"os--R\'enyi random graph,
and the uniform multinomial occupancy model. In Section~\ref{sec3}, we give the
proof of Theorem~\ref{6-t1}.

\section{Applications}\label{sec2}

In this section, we apply Theorem~\ref{6-t1} to prove discretized
normal approximation results for integer valued random variables with
different dependence structures including local dependence,
exchangeable pairs, and size biasing.

\subsection{Local dependence}
A typical setting of local dependence is as follows. Let $S=\sum_{i=1}^n X_i$ be a sum of integer valued random variables with $\E
X_i=\mu_i$, $\mu=\sum_{i=1}^n \mu_i$ and $\Var(S)=\sigma^2$.
Suppose for each $i\in\{1,\ldots,n\}$, there exist neighborhoods
$A_i, B_i \subset\{1,\ldots,n\}$ such that $X_i$ is independent of $\{
X_j: j\notin A_i\}$, and $\{X_j: j\in A_i\}$ is independent of $\{X_j:
j\notin B_i\}$. It can be verified as in Section~3.2 of \cite{ChRo10} that\vspace*{-2.2pt}
\begin{eqnarray*}
\bigl(S,S', G\bigr)=\biggl(S,S-\sum_{j\in A_I}(X_j-
\mu_j), -n(X_I-\mu_I)\biggr)
\end{eqnarray*}
is a Stein coupling where $I$ is a uniform random index from $\{
1,\ldots,n\}$ and independent of $\{X_1,\ldots, X_n\}$. Theorem~\ref
{6-t1} has the following corollary for local dependence.
%
\begin{cor}\label{6-t2}
Under the above setting, assume that for every $i\in\{1,\ldots,n\}$,
$|N(B_i)|\le\theta$ where $N(B_i)=\{j\in\{1,\ldots, n\}: A_j\cap
B_i\ne\emptyset\}$ and $|\cdot|$ denotes cardinality. Let\vspace*{-2.2pt}
\begin{eqnarray*}
\xi_i=\frac{X_i-\mu_i}{\sigma},\qquad  \eta_i=\sum
_{j\in A_i}\xi_j.
\end{eqnarray*}
Then,\vspace*{-2.2pt}
%
\begin{eqnarray}
\label{6-t2-1} &&d_{\mathrm{TV}}\bigl(\mathcal{L}(S), N^d\bigl(\mu,
\sigma^2\bigr)\bigr)\nonumber
\\[-2.2pt]
&&\quad \le2\sqrt{\theta\sum_{i=1}^n \E
\xi_i^2 \eta_i^2}+\sqrt{
\frac
{\pi}{8}} \sum_{i=1}^n \E\bigl|
\xi_i \eta_i^2\bigr| + \sqrt{n \sum
_{i=1}^n \E\xi_i^2
\eta_i^4}
\\[-2.2pt]
&&\qquad {} +\frac{1}{2}\sum_{i=1}^n \E
\bigl[ \bigl( \sigma\bigl|\xi_i \eta_i^2\bigr| + |
\xi_i \eta_i|\bigr) d_{\mathrm{TV}} \bigl(
\mathcal{L}(S|\mathcal{F}_i), \mathcal{L}( S+1|\mathcal{F}_i)
\bigr) \bigr],\nonumber
\end{eqnarray}
where $\mathcal{F}_i$ is a $\sigma$-field such that $\sigma(X_j:
j\in A_i)\subset\mathcal{F}_i$.
\end{cor}
\begin{pf}
Let $I$ be a uniform random index from $\{1,\ldots,n\}$ and
independent of $\{X_1,\ldots, X_n\}$. Let $G=-n(X_I-\mu_I)$, $D=-\sum_{j\in A_I} (X_j-\mu_j)$, and let $X_{A_i}=\{X_j: j\in A_i\}$. We
bound the right-hand side of (\ref{6-t1-1}) as follows. From the
definition of neighborhoods $A_i, B_i$, the inequality $\cov(X,Y) \le
(\E X^2+\E Y^2)/2$ and the bound $|N(B_i)|\leq\theta$, we have
\begin{eqnarray*}
&&\Var\bigl(\E(GD|S)\bigr)\le\Var\bigl(\E\bigl(GD|\{X_1,\dots,
X_n\}\bigr)\bigr) \\
&&\quad = \Var\Biggl( \sum_{i=1}^n
(X_i-\mu_i) \sum_{j\in A_i}
(X_j-\mu_j)\Biggr)
\\
&&\quad \le\sum_{i,i': X_{A_i}, X_{A_{i'}} \mathrm{not\ independent}} \cov \biggl( (X_i-
\mu_i)\sum_{j\in A_i}(X_j-
\mu_j), (X_{i'}-\mu _{i'})\sum
_{j'\in A_{i'}}(X_{j'}-\mu_{j'}) \biggr)
\\
&&\quad \le\sum_{i,i': X_{A_i}, X_{A_{i'}} \mathrm{not\ independent}} \biggl\{ \frac{\E[(X_i-\mu_i)\sum_{j\in A_i}(X_j-\mu_j)]^2}{2}
\\
&&\quad \hphantom{\le\sum_{i,i': X_{A_i}, X_{A_{i'}} \mathrm{not\ independent}} \biggl\{} +\frac{\E[(X_{i'}-\mu_{i'})\sum_{j'\in
A_{i'}}(X_{j'}-\mu_{j'})]^2 }{2} \biggr\}
\\
&&\quad \le\theta\sum_{i=1}^n \E
\biggl[(X_i-\mu_i)\sum_{j\in A_i}(X_j-
\mu _j)\biggr]^2
\\
&&\quad = \sigma^4 \theta\sum_{i=1}^n
\E\xi_i^2 \eta_i^2.
\end{eqnarray*}
Moreover,
\begin{eqnarray*}
\E|GD|=\sigma^2 \sum_{i=1}^n
\E|\xi_i \eta_i|,\qquad  \E\bigl|GD^2\bigr|= \sigma
^3 \sum_{i=1}^n \E\bigl|
\xi_i \eta_i^2\bigr|, \qquad \E G^2D^4
=n \sigma^6 \sum_{i=1}^n \E
\xi_i^2 \eta_i^4.
\end{eqnarray*}

The corollary is proved by applying the above bounds in (\ref{6-t1-1})
with $\mathcal{F}=\sigma(I,\mathcal{F}_I)$.
\end{pf}

We remark that in the case that $S$ is a sum of independent integer
valued random variables, a modification of the arguments from
intermediate terms in the proof of Theorem~\ref{6-t1} yields a result
similar to Theorem~7.4 of \cite{ChGoSh11}.

\subsubsection{$2$-runs}
We provide a concrete example of local dependence here. Let $\zeta
_1,\ldots, \zeta_{n}$ be independent and identically distributed
Bernoulli variables with $\P(\zeta_1=1)=1-\P(\zeta_1=0)=p$ where
$p\in(0,1)$. Suppose $n\ge7$. Let $X_i=\zeta_i \zeta_{i+1}$ and
$S=\sum_{i=1}^n X_i$. Here and in the rest of this example, indices
outside $\{1,\dots, n\}$ are understood as one plus their residues mod
$n$. We can apply Corollary~\ref{6-t2} with $A_i=\{i-1, i, i+1\}$,
$B_i=\{i-2,\dots, i+2\}$, so that $\theta=7$. The mean and variance
of $S$ can be calculated as
%
\begin{equation}
\label{2runs10} \mu=\E S=np^2, \qquad \sigma^2=\Var(S)=n
\bigl(p^2+2p^3-3p^4\bigr).
\end{equation}
Applying (\ref{6-t2-1}) with $\mathcal{F}_i=\sigma(\zeta
_{i-1},\zeta_i,\zeta_{i+1},\zeta_{i+2})$, along with the upper
bounds $|\xi_i|\le1/\sigma, |\eta_i|\le3/\sigma$, we have
\begin{eqnarray*}
d_{\mathrm{TV}}\bigl(\mathcal{L}(S), N^d\bigl(\mu,
\sigma^2\bigr)\bigr) \le c_p'
\frac{1}{\sqrt {n}}+c_p''\sup
_{a,b\in\{0,1\}} d_{\mathrm{TV}}\bigl(\mathcal{L}(V_{a,b}),
\mathcal {L}(V_{a,b}+1)\bigr),
\end{eqnarray*}
where $c_p', c_p''$ are constants depending on $p$ and with $m=n-4$ and
$a,b \in\{0,1\}$ given,
\begin{eqnarray*}
V_{a,b}=a \zeta_1+\sum_{j=2}^{m}
\zeta_{j-1} \zeta_{j} +b\zeta_m.
\end{eqnarray*}
Regarding $V_{a,b}=f(\zeta_1,\ldots, \zeta_m)$, we define
$V_{a,b}'=f(\zeta_1,\ldots, \zeta_I',\ldots, \zeta_m)$ where $I$
is uniformly chosen from $\{1,\ldots,m\}$, independent of $\{\zeta
_1,\ldots, \zeta_{m}\}$ and given $I$, $\zeta_I'$ is an independent
copy of $\zeta_I$. From Remark~\ref{exchconst}, $(V_{a,b},V_{a,b}')$
is an exchangeable pair.
Since given $\{\zeta_1,\dots,\zeta_m\}$ and \mbox{$I=i$},
\begin{eqnarray*}
\bigl\{V_{a,b}-V_{a,b}'=1\bigr\}=
\cases{ \bigl\{a+\zeta_2=1,\zeta_1=1,\zeta_1'=0
\bigr\}, &\quad  $i=1$,
\vspace*{1pt}\cr
\bigl\{b+\zeta_{m-1}=1, \zeta_m=1,
\zeta_m'=0\bigr\}, &\quad  $i=m$,
\vspace*{1pt}\cr
\bigl\{
\zeta_{i-1}+\zeta_{i+1}=1, \zeta_i=1,
\zeta_i'=0\bigr\}, &\quad  $2\leq i\leq m-1$, } %
\end{eqnarray*}
we have
%
\begin{eqnarray}
\label{131} && \E\bigl(I\bigl(V_{a,b}-V_{a,b}'=1
\bigr)|\{\zeta_1,\dots, \zeta_m\}\bigr)\nonumber
\\
&&\quad =\frac{1-p}{m} \Biggl[I(a+\zeta_2=1, \zeta_1=1)+I(b+
\zeta_{m-1}=1, \zeta_m=1)\\
&&\hphantom{\quad =\frac{1-p}{m} \Biggl[}{}+\sum_{i=2}^{m-1}
I(\zeta_{i-1}+\zeta_{i+1}=1, \zeta _i=1) \Biggr].\nonumber
\end{eqnarray}
Taking expectation on both sides of \eqref{131} and lower bounding the
right-hand side by the last term lead to
\begin{eqnarray*}
\P\bigl(V_{a,b}-V_{a,b}'=1\bigr)\ge
\frac{2(n-6)}{n-4} p^2(1-p)^2.
\end{eqnarray*}
In calculating the variance of the right-hand side of \eqref{131}, we use
the fact that each indicator is only correlated with at most two other
indicators. Therefore,
\begin{eqnarray*}
 \sqrt{\Var\bigl(\E\bigl(I\bigl(V_{a,b}-V_{a,b}'=1
\bigr)|V_{a,b}\bigr)\bigr)}&\le&\sqrt{\Var\bigl( \E \bigl(I
\bigl(V_{a,b}-V_{a,b}'=1\bigr)|\{
\zeta_1,\dots, \zeta_m\}\bigr) \bigr) }
\\
& \le&\frac{1-p}{n-4} \sqrt{3(n-4)}.
\end{eqnarray*}
Similarly,\vspace*{-1pt}
\begin{eqnarray*}
\sqrt{\Var\bigl(\E\bigl(I\bigl(V_{a,b}-V_{a,b}'=-1
\bigr)|V_{a,b}\bigr)\bigr)}\le\frac{p}{n-4} \sqrt{3(n-4)}.
\end{eqnarray*}
Applying Lemma~\ref{6-l1}, we have\vspace*{-1pt}
\begin{eqnarray*}
d_{\mathrm{TV}}\bigl(\mathcal{L}(V_{a,b}),\mathcal{L}(V_{a,b}+1)
\bigr)\le\frac{\sqrt {3(n-4)}}{2(n-6)p^2(1-p)^2}.
\end{eqnarray*}
Therefore, we have proved the following proposition.
%
\begin{prop}
For $n\ge2$, let $\zeta_1,\ldots, \zeta_{n}$ be independent and
identically distributed Bernoulli variables with $\P(\zeta_1=1)=1-\P
(\zeta_1=0)=p$ where $p\in(0,1)$. Let $X_i=\zeta_i \zeta_{i+1}$ and
$S=\sum_{i=1}^n X_i$. We have\vspace*{-1pt}
%
\begin{equation}
\label{2runs0} d_{\mathrm{TV}}\bigl(\mathcal{L}(S), N^d\bigl(\mu,
\sigma^2\bigr)\bigr) \le c_p/\sqrt{n},
\end{equation}
where $\mu$ and $\sigma^2$ are defined as in \eqref{2runs10} and $c_p$
is a constant depending on $p$.
\end{prop}

We remark that the above argument also applies to $k$-runs for $k>2$
with straightforward modifications, for example, enlarging the
neighborhoods $A_i$ and $B_i$, changing the definition of $\mathcal
{F}_i$, etc.

Total variation approximation for $2$-runs was studied by Barbour and
Xia \cite{BaXi99} and R\"ollin \cite{Ro05} using the translated
Poission approximation. Barbour and Xia \cite{BaXi99} assumed some
extra conditions on $p$ to obtain a bound on the total variation
distance between $\mathcal{L}(S)$ and a translated Poisson
distribution. Although the result in \cite{Ro05} is of the same order
as the bound in \eqref{2runs0} in terms of $n$ and applies for all $p$,
the approach used was different from ours.

\subsection{Exchangeable pairs}
A systematic introduction on the exchangeable pair approach can be
found in Stein \cite{St86}. The basic setting is as follows. Let
$(S,S')$ be an exchangeable pair (i.e., $\mathcal{L}(S,S')=\mathcal
{L}(S',S)$) of integer valued random variables with $\E S= \mu$, $\Var
(S)= \sigma^2$. Suppose we have the approximate linearity condition,\vspace*{-1pt}
%
\begin{equation}
\label{6.3.2-1} \E\bigl(S-S'|S\bigr)=\lambda(S-\mu)+\sigma\E(R|S),
\end{equation}
for a positive number $\lambda$ and a random variable $R$.
A simple modification of Theorem~\ref{6-t1} yields the following
corollary for exchangeable pairs.
%
\begin{cor}\label{6-c1}
Let $(S,S')$ be an exchangeable pair of integer valued random variables
satisfying \eqref{6.3.2-1}. Let $\mu=\E S, \sigma^2=\Var(S)$.
We have
%
\begin{eqnarray}
\label{6-c1-1} &&d_{\mathrm{TV}}\bigl(\mathcal{L}(S), N^d \bigl(\mu,
\sigma^2\bigr)\bigr)\nonumber
\\
&&\quad \le\biggl(\sqrt{\frac{\pi}{2}}+2\biggr)\frac{\sqrt{\E R^2}}{\lambda}+
\frac
{\sqrt{\Var(\E((S'-S)^2|S))}}{\lambda\sigma^2}\\
&&\qquad {}+\sqrt{\frac{\pi
}{8}} \frac{\E|S'-S|^3}{2\lambda\sigma^3} +\frac{ \sqrt{\E
|S'-S|^6} }{2\lambda\sigma^3}\nonumber
\\
&&\qquad {} +\frac{1}{4\lambda\sigma^2}\E \bigl[ \bigl(\bigl|S'-S\bigr|^3+
\bigl(S'-S\bigr)^2\bigr) d_{\mathrm{TV}}\bigl(
\mathcal{L}(S|\mathcal{F}), \mathcal{L}(S+1|\mathcal{F})\bigr) \bigr],\nonumber
\end{eqnarray}
where $\mathcal{F}$ is a $\sigma$-field such that $\sigma
(S'-S)\subset\mathcal{F}$.
\end{cor}
\begin{pf}
We follow the proof of Theorem~\ref{6-t1} with minor modification.
Let $G=\frac{1}{2\lambda} (S'-S)$ and $D=S'-S$.
From the exchangeability of $(S,S')$,
\begin{eqnarray*}
\E G \bigl(f\bigl(S'\bigr)+f(S)\bigr)=0.
\end{eqnarray*}
By (\ref{6.3.2-1}) and the above equality,
\begin{eqnarray*}
\E(S-\mu) f(S)=\E\bigl\{G f\bigl(S'\bigr)-G f(S)\bigr\}-
\frac{\sigma}{\lambda} \E f(S) R.
\end{eqnarray*}
Therefore, (\ref{6-t1-7}) has an extra term $\sigma\E f_h(S)
R/\lambda$, which is bounded by $\sqrt{\pi/2}\E|R|/\lambda$ from
\eqref{6-t1-5}.
Moreover, from the exchangeability of $(S,S')$ and \eqref{6.3.2-1},
\begin{eqnarray*}
\E G D &=& \frac{1}{2\lambda} \E\bigl(S'-S\bigr)^2
\\
&=&\frac{1}{2\lambda} \bigl[ \E\bigl(S'-S\bigr) S' -
\E\bigl(S'-S\bigr)S \bigr]
\\
&=&\frac{1}{\lambda} \E\bigl(S-S'\bigr)S=\frac{1}{\lambda} \E
\bigl(S-S'\bigr) (S-\mu)
\\
&=&\sigma^2+\sigma\E\bigl((S-\mu)R\bigr)/\lambda.
\end{eqnarray*}
Hence instead of \eqref{6-t1-10},
\begin{eqnarray*}
|R_1|\le\frac{2}{\sigma^2}\biggl(\sqrt{\Var\bigl(\E(GD|S)\bigr)}+
\frac{\sigma
}{\lambda}\E\bigl|(S-\mu)R\bigr| \biggr)\le\frac{\sqrt{\Var(\E
((S'-S)^2|S))}}{\lambda\sigma^2}+\frac{2}{\lambda}
\sqrt{\E R^2}.
\end{eqnarray*}
Corollary~\ref{6-c1} follows from Theorem~\ref{6-t1} and the above arguments.
\end{pf}
A special case worth mentioning is when the exchangeable pair $(S,S')$
satisfies $|S-S'|\le1$. Examples of such exchangeable pairs include
binary expansion of a random integer [Diaconis \cite{Di77}] and
anti-voter model [Rinott and Rotar \cite{RiRo97}]. The following
result shows that under this special assumption, bounding the total
variation distance requires no more effort than bounding the Kolmogorov
distance.
%
\begin{cor}\label{6-c2}
Let $(S,S')$ be an exchangeable pair of integer valued random variables
satisfying the approximate linearity condition (\ref{6.3.2-1}). In
addition, suppose $|S-S'|\le1$. Then we have
%
\begin{eqnarray}
\label{6-c2-1} &&d_{\mathrm{TV}}\bigl(\mathcal{L}(S), N^d \bigl(\mu,
\sigma^2\bigr)\bigr)
\nonumber\\[-8pt]\\[-8pt]
&&\quad \le\biggl(\sqrt{\frac{\pi}{2}}+2\biggr)\frac{\sqrt{\E R^2}}{\lambda}+
\frac
{\sqrt{\Var(\E((S'-S)^2|S))}}{\lambda\sigma^2}+ \frac{\sqrt{\pi
/8}+1}{2\lambda\sigma^3},\nonumber
\end{eqnarray}
where $\mu$ and $\sigma^2$ are the mean and variance of $S$.
\end{cor}
\begin{pf}
Let $G=\frac{1}{2\lambda}(S'-S)$, $D=S'-S$. Then for $h\in\mathcal
{H}$ defined in \eqref{6-t1-3},
\begin{eqnarray}
\label{6-c2-2} && \E G\int_0^D\bigl(h(S+t)-h(S)
\bigr)\,\mathrm{d}t\nonumber
\\
&&\quad =\frac{1}{2\lambda} \E\bigl(S'-S\bigr)\int_0^{S'-S}
\bigl(h(S+t)-h(S)\bigr)\,\mathrm{d}t\nonumber
\\
&&\quad =\frac{1}{2\lambda} \E \biggl[ \int_0^1
\bigl(h(S+t)-h(S)\bigr)\,\mathrm{d}t I\bigl(S'-S=1\bigr)\nonumber
\\
&&\qquad\hphantom{\frac{1}{2\lambda} \E \biggl[} {} -\int_0^{-1} \bigl(h(S+t)-h(S)
\bigr)\,\mathrm{d}tI\bigl(S'-S=-1\bigr) \biggr]
\\
&&\quad =\frac{1}{4\lambda} \E \bigl[ \bigl(h(S+1)-h(S)\bigr)I\bigl(S'-S=1
\bigr)+\bigl(h(S-1)-h(S)\bigr)I\bigl(S'-S=-1\bigr) \bigr]\nonumber
\\
&&\quad =\frac{1}{4\lambda}\E \bigl[ \bigl(h\bigl(S'\bigr)-h(S)\bigr)I
\bigl(S'-S=1\bigr)-\bigl(h(S)-h\bigl(S'\bigr)\bigr)I
\bigl(S-S'=1\bigr) \bigr]\nonumber
\\
&&\quad =0.\nonumber
\end{eqnarray}
We used the exchangeability of $(S,S')$ in the last equality. From
(\ref{6-c2-2}), the upper bound in \eqref{1101} can be replaced by $0$.
Therefore, the bound on $d_{\mathrm{TV}}(\mathcal{L}(S), N^d (\mu, \sigma
^2))$ can be deduced similarly as Corollary~\ref{6-c1} except that we
do not have the last term on the right-hand side of \eqref{6-c1-1}.
\end{pf}
%
\begin{rem}
Under the condition of Corollary~\ref{6-c2}, R\"ollin \cite{Ro07}
obtained a bound on the total variation distance between $\mathcal
{L}(S)$ and a translated Poisson distribution. His result, together
with the triangle inequality and easy bounds on the total variation
distance between the translated Poisson distribution and the
discretized normal distribution, yields a similar bound as \eqref{6-c2-1}.
\end{rem}

\subsection{Size biasing}
Size biasing was first introduced in the context of Stein's method by
Goldstein and Rinott \cite{GoRi96}.
For $S$ being a nonnegative integer valued random variable with mean
$\mu$, we say $S^s$ has the $S$-size biased distribution if
\begin{eqnarray*}
\E Sf(S) =\E\mu f\bigl(S^s\bigr)
\end{eqnarray*}
for all $f$ such that the above expectations exist. If in addition
$S^s$ is defined on the same probability space as $S$, then
%
\begin{equation}
\label{6.1-5} \bigl(S,S',G\bigr)=\bigl(S, S^s, \mu
\bigr)
\end{equation}
is a Stein coupling.
Theorem~\ref{6-t1} has the following corollary for size biasing which
easily follows from \eqref{6.1-5}.
%
\begin{cor}\label{6-c3}
Let $S$ be a nonnegative integer valued random variable with mean $\mu
$ and variance $\sigma^2$. Let $S^s$ be defined on the same
probability space and have the $S$-size biased distribution. Then
%
\begin{eqnarray}
\label{6-c3-1} &&d_{\mathrm{TV}}\bigl(\mathcal{L}(S), N^d\bigl(\mu,
\sigma^2\bigr)\bigr)\nonumber
\\
&&\quad \le\frac{2\mu}{\sigma^2} \sqrt{\Var\bigl(\E\bigl(S^s-S|S\bigr)
\bigr)}+\sqrt{\frac
{\pi}{8}} \frac{\mu}{\sigma^3} \E\bigl|S^s-S\bigr|^2+
\frac{\mu}{\sigma^3} \sqrt{\E\bigl|S^s-S\bigr|^4}
\\
&&\qquad {} +\frac{\mu}{2\sigma^2} \E \bigl[ \bigl( \bigl|S^s-S\bigr|^2+\bigl|S^s-S\bigr|
\bigr) d_{\mathrm{TV}} \bigl(\mathcal{L}(S|\mathcal{F}),\mathcal{L}(S+1|
\mathcal{F})\bigr) \bigr],\nonumber
\end{eqnarray}
where $\mathcal{F}$ is a $\sigma$-field such that $\sigma
(S^s-S)\subset\mathcal{F}$.
\end{cor}
Next, we apply Corollary~\ref{6-c3} to bound the total variation
distance for discretized normal approximations for the number of
vertices with a given degree in the Erd\"os--R\'enyi random graph, and
the uniform multinomial occupancy model. These two models were recently
studied by Goldstein \cite{Go12} and Bartroff and Goldstein \cite
{BaGo12}, respectively. They obtained the same bound for the Kolmogorov
distance using the inductive size bias coupling technique introduced by
Goldstein \cite{Go12}.

\subsubsection{Number of vertices with a given degree in the Erd\"
os--R\'enyi random graph}
Let $G(n,p_n)$ be an Erd\"os--R\'enyi random graph with vertex set $\{
1,\dots,n\}$ and edge probability $p_n\in(0,1)$. Let $S_n$ be the
number of vertices with a given degree $d\ge0$ in $G(n, p_n)$.
The asymptotic normality of $S_n$ was proved in \cite{BaKaRu89} when
$n p_n \rightarrow\theta>0$. Under the condition
%
\begin{eqnarray}
\label{2.1} &&\mbox{there exist } 0<\theta'\le\theta_n
\le\theta''<\infty, n_0>0 \mbox{ such
that}
\nonumber
\\[-8pt]\\[-8pt]
&&\quad p_n=\theta_n/(n-1)\qquad  \mbox{for all } n\ge
n_0,\nonumber
\end{eqnarray}
Goldstein \cite{Go12} proved a bound on the Kolmogorov distance
between the distribution of $S_n$ and $N(\mu_n, \sigma_n^2)$,
\begin{eqnarray*}
d_K\bigl(\mathcal{L}(S_n), N\bigl(\mu_n,
\sigma_n^2\bigr)\bigr) \le c_d/\sqrt{n},
\end{eqnarray*}
where $\mu_n$ and $\sigma_n^2$ are the mean and variance of $S_n$,
respectively. Here and in the rest of this example, let $c_d=c(d,\theta
',\theta'',n_0)$ denote positive constants which may depend on
$d,\theta',\theta'',n_0$. In the following proposition, we prove a
bound on the total variation distance between the distribution of $S_n$
and $N^d(\mu_n, \sigma_n^2)$.
%
\begin{prop}\label{prop3}
Let $G(n,p_n)$, $n\geq2$, be a sequence of Erd\"os--R\'enyi random
graphs satisfying \eqref{2.1}.
Let $S_n$ be the number of vertices with a given degree $d$ in $G(n,
p_n)$. We have
%
\begin{equation}
\label{prop3-1} d_{\mathrm{TV}}\bigl(\mathcal{L}(S_n),
N^d\bigl(\mu_n, \sigma_n^2\bigr)
\bigr) \le c_d/\sqrt{n}.
\end{equation}
\end{prop}
\begin{pf}
Since the total variation distance is always bounded by $1$,
for $n<\max\{n_0, 8\}$, \eqref{prop3-1} holds true by choosing
$c_d=\max
\{\sqrt{n_0}, 2\sqrt{2}\}$. Therefore, we assume $n\ge\max\{n_0, 8\}
$ in the rest of the proof.

In \cite{Go12}, it was proved that under condition \eqref{2.1},
%
\begin{equation}
\label{132}  \frac{n}{c_d} \leq \mu_n \leq c_d n, \qquad  \frac{n}{c_d} \leq \sigma_n^2 \leq c_d n.
\end{equation}
Let $\deg(i)$ denote the degree of vertex $i$. Then $S_n$ can be
expressed as
\begin{eqnarray*}
S_n=\sum_{i=1}^n I
\bigl(\deg(i)=d\bigr).
\end{eqnarray*}
Following the construction of size bias coupling in Goldstein and
Rinott \cite{GoRi96}, let $I$ be uniformly chosen from $\{1,\dots,n\}
$ and independent of $G(n,p_n)$. If $\deg(I)=d$, then we define
$G^s(n,p_n)$, the size biased graph, to be the same as $G(n,p_n)$. If
$\deg(I)>d$, then we obtain $G^s(n,p_n)$ from $G(n,p_n)$ by removing
$\deg(I)-d$ edges chosen uniformly at random from the edges that connect
to $I$ in $G(n,p_n)$. If $\deg(I)<d$, then we obtain $G^s(n,p_n)$ from
$G(n,p_n)$ by connecting $I$ to $d-\deg(I)$ vertices chosen uniformly at
random from those not connected to $I$ in $G(n, p_n)$. Let $S_n^s$ be
the number of vertices with degree $d$ in the graph $G^s(n,p_n)$. It
was proved in \cite{GoRi96} that $S_n^s$ has the $S_n$-size biased
distribution and
%
\begin{equation}
\label{133} \Var\bigl(\E\bigl(S_n^s-S_n|S_n
\bigr)\bigr)\le c_d/n.
\end{equation}
From the construction of $G^s(n,p_n)$, at most $|\deg(I)-d|+1$ vertices
have different degrees in $G(n,p_n)$ and $G^s(n,p_n)$. Therefore,
%
\begin{equation}
\label{2.2} \bigl|S_n^s-S_n\bigr|\le\bigl|\deg(I)-d\bigr|+1.
\end{equation}
Given $I$, $\deg(I)\sim \Binomial(n-1,p_n)$. This, together with \eqref
{2.1}, implies that for any positive integer $k\leq4$,
%
\begin{equation}
\label{2.2a} \E \deg(I)^k \leq c_d.
\end{equation}
From \eqref{2.2} and \eqref{2.2a},
%
\begin{equation}
\label{134} \E\bigl|S_n^s-S_n\bigr|^k
\le c_d,\qquad  k\leq4.
\end{equation}
Applying \eqref{132}, \eqref{133} and \eqref{134} in \eqref
{6-c3-1}, the proof
will be complete after we show that
%
\begin{equation}
\label{2.3} \E \bigl[ \bigl(\bigl|S_n^s-S_n\bigr|^2+\bigl|S_n^s-S_n\bigr|
\bigr)d_{\mathrm{TV}}\bigl(\mathcal{L}(S_n|\mathcal {F}),
\mathcal{L}(S_n+1|\mathcal{F})\bigr) \bigr]\le c_d/
\sqrt{n}
\end{equation}
for a $\sigma$-field $\mathcal{F}$ such that $\sigma
(S_n^s-S_n)\subset\mathcal{F}$.
For a given $I$, define
\begin{eqnarray*}
A_I=\{I\} \cup\bigl\{j: e_{Ij}=1 \mbox{ or }
e^s_{Ij}=1 \bigr\},\qquad  B_I=\{k\notin
A_I: e_{kj}=1 \mbox{ for some} j\in A_I\},
\end{eqnarray*}
where $e_{uv}$ ($e^s_{uv}$) is the indicator that there is an edge
connecting $u$ and $v$ in $G(n,p_n)$ ($G^s(n,p_n)$). Let
%
\begin{equation}
\label{2.4} \mathcal{F}=\sigma\bigl(I, A_I, B_I,
\{e_{uv}: u\in A_I, v\in A_I\cup
B_I\} , \bigl\{e^s_{Iv}: v\in
A_I\bigr\}\bigr).
\end{equation}
From the construction of $G^s(n,p_n)$, we have $\sigma
(S^s_n-S_n)\subset\mathcal{F}$. Let $|\cdot|$ denote cardinality
when the argument is a set. From \eqref{2.2}, \eqref{2.2a} and
$|A_I|= \max
(\deg(I),d)+1$,
\begin{eqnarray*}
&&\E\bigl(\bigl|S_n^s-S_n\bigr|^2+\bigl|S_n^s-S_n\bigr|
\bigr)I\bigl(|A_I|>\sqrt{n}\bigr)
\\
&&\quad \le2\E|A_I|^2 I\bigl(|A_I|>\sqrt{n}\bigr)\le
\frac{2}{\sqrt{n}} \E|A_I|^3
\\
&&\quad = \frac{2}{\sqrt{n}}\E\bigl(\max\bigl(\deg(I),d\bigr)+1\bigr)^3
\\
&&\quad \le c_d/\sqrt{n}.
\end{eqnarray*}
Similarly,
\begin{eqnarray*}
&&\E\bigl(\bigl|S_n^s-S_n\bigr|^2+\bigl|S_n^s-S_n\bigr|
\bigr)I\bigl(|B_I|>\sqrt{n}\bigr)\\
&&\quad  \le2\E|A_I|^2
|B_I|/\sqrt{n}
\le 2\E|A_I|^2 \bigl[ \E\bigl(|B_I||I,
A_I\bigr) \bigr]/\sqrt{n}\le c_d \E |A_I|^3/
\sqrt{n}\le c_d/\sqrt{n},
\end{eqnarray*}
where we used $\E(|B_I||I, A_I)\leq c_d|A_I|$, which is from the fact
that the expected degree of a given vertex is bounded by $c_d$ under
condition \eqref{2.1}.
Therefore, to prove \eqref{2.3}, we only need to prove
%
\begin{eqnarray}
\label{2.5}
&&\E \bigl[ \bigl(\bigl|S_n^s-S_n\bigr|^2+\bigl|S_n^s-S_n\bigr|
\bigr)I\bigl(|A_I|, |B_I|\le\sqrt {n}\bigr)d_{\mathrm{TV}}\bigl(
\mathcal{L}(S_n|\mathcal{F}), \mathcal{L}(S_n+1|\mathcal
{F})\bigr) \bigr]\nonumber\\[-8pt]\\[-8pt]
&&\quad \le c_d/\sqrt{n},\nonumber
\end{eqnarray}
where $\mathcal{F}$ was defined in \eqref{2.4}. Given $\mathcal{F}$
with $|A_I|, |B_I|\leq\sqrt{n}$, we define a random graph
$G^{\mathcal{F}}$ with vertex set $\{1,\dots, n\}$ by letting
$e^{\mathcal{F}}_{uv}=e_{uv}$ for $u\in A_I, v\in\{1,\dots, n\}$,
and letting $e^{\mathcal{F}}_{uv}$ be independent $\operatorname{Bernoulli}(p_n)$
random variables for $u,v \in(A_I)^c$ where $e^{\mathcal{F}}$ is the
edge indicator for $G^{\mathcal{F}}$. Let $V^{\mathcal{F}}=\sum_{i=1}^n I(\deg^{\mathcal{F}}(i)=d)$ be the number of vertices with
degree $d$ in $G^{\mathcal{F}}$. Then $\mathcal{L}(V^{\mathcal
{F}})=\mathcal{L}(S_n|\mathcal{F})$, which follows from $\mathcal
{L}(G^{\mathcal{F}})=\mathcal{L}(G(n,p_n)|\mathcal{F})$.

In the following we fix a given $\mathcal{F}$ with $|A_I|, |B_I|\leq
\sqrt{n}$, and prove
%
\begin{equation}
\label{2.5a} d_{\mathrm{TV}} \bigl(\mathcal{L}\bigl(V^{\mathcal{F}}\bigr),
\mathcal{L}\bigl(V^{\mathcal
{F}}+1\bigr)\bigr)\le c_d/\sqrt{n}.
\end{equation}
For ease of notation, we suppress the superscript $\mathcal{F}$, that
is, let $G=G^{\mathcal{F}}, V=V^{\mathcal{F}}, e=e^\mathcal{F},
\deg=\deg^{\mathcal{F}}$.
To bound $d_{\mathrm{TV}}(\mathcal{L}(V),\mathcal{L}(V+1))$ using Lemma~\ref
{6-l1} and Remark~\ref{exchconst}, we construct an exchangeable pair
$(V,V')$ by uniformly choosing $J\ne K$ from $C_I:=(A_I\cup B_I)^c$ and
independently resampling $e_{JK}$ to be $e'_{JK}$. Writing out all four
possibilities for $\{V-V'=1\}$,
%
\begin{eqnarray}
\label{135} I\bigl(V-V'=1\bigr)&=&(1-e_{JK})e'_{JK}
\bigl\{ I\bigl(\deg(J)=d\bigr)I\bigl(\deg(K)\ne d-1,d\bigr)\nonumber
\\
&&\hphantom{(1-e_{JK})e'_{JK}
\bigl\{}{}+I\bigl(\deg(J)\ne d-1,d\bigr)I\bigl(\deg(K)=d\bigr) \bigr\}\nonumber
\\[-8pt]\\[-8pt]
&&{}+e_{JK}\bigl(1-e'_{JK}\bigr) \bigl
\{I\bigl(\deg(J)=d\bigr)I\bigl(\deg(K)\ne d,d+1\bigr)\nonumber
\\
&&\hphantom{+e_{JK}\bigl(1-e'_{JK}\bigr) \bigl
\{}{}+I\bigl(\deg(J)\ne d,d+1\bigr)I\bigl(\deg(K)=d\bigr) \bigr\}.\nonumber
\end{eqnarray}
Let $m=|C_I|\ge n-2\sqrt{n}\ge2$ (recall $n\ge8$), and let $\xi_1,
\xi_2$ be independent $\Binomial(|B_I|+m-2,p_n)$ random variables.
Taking expectation on both sides of \eqref{135}, lower bounding the
right-hand side by its first term and observing that $\deg(J)$ and
$\deg(K)$ are independent $\sim \Binomial(|B_I|+m-2, p_n)$ given that $J$
is not connected to $K$ and $J, K\in C_I$ (thus not connected to
$A_I$), we have
\begin{eqnarray*}
\P\bigl(V-V'=1\bigr)&\ge& \E(1-e_{JK})
e'_{JK} I\bigl(\deg(J)=d\bigr)I\bigl(\deg(K)\ne d-1,d
\bigr)
\\
&=& \frac{1}{m(m-1)} \sum_{j, k \in C_I: j\ne k}(1-p_n)p_n
\P(\xi _1=d)\P(\xi_2\ne d-1,d).
\end{eqnarray*}
From \eqref{2.1} and $n-2\sqrt{n}\leq|B_I|+m-2\leq n$, we have, for
some positive constant $c_d$,
\begin{eqnarray*}
p_n\geq\frac{c_d}{n},\qquad  1-p_n\geq c_d,\qquad
\P(\xi_1=d)\geq c_d, \qquad \P(\xi _2\ne d-1, d)\geq
c_d.
\end{eqnarray*}
Therefore,
%
\begin{equation}
\label{136} \P\bigl(V-V'=1\bigr)\ge c_d/n.
\end{equation}
Next, we obtain an upper bound of $\Var(\E(I(V-V'=1)|V))$. By taking
expectation with respect to $J, K$ first and then writing the variance
of a sum as a sum of covariances, we have
\begin{eqnarray*}
&&\Var\bigl(\E\bigl((1-e_{JK})e'_{JK}I
\bigl(\deg(J)=d\bigr)I\bigl(\deg(K)\ne d-1,d\bigr)|V\bigr)\bigr)
\\
&&\quad \le\Var\bigl(\E\bigl((1-e_{JK})e'_{JK}I
\bigl(\deg(J)=d\bigr)I\bigl(\deg(K)\ne d-1,d\bigr)|G, \mathcal{F}\bigr)\bigr)
\\
&&\quad \le\frac{c_d}{n^4}\Var \biggl[ \sum_{j,k\in C_I: j\ne
k}(1-e_{jk})e'_{jk}I
\bigl(\deg(j)=d\bigr)I\bigl(\deg(k)\ne d-1,d\bigr) \biggr]
\\
&&\quad =\frac{c_d}{n^4} \sum_{j,k,j',k'\in C_I: \atop j\ne k, j'\ne k', |\{
j,k,j',k'\}|=2}\cov \bigl[
(1-e_{jk})e'_{jk}I\bigl(\deg(j)=d\bigr)I
\bigl(\deg(k)\ne d-1,d\bigr),
\\
&&\hphantom{\quad =\frac{c_d}{n^4} \sum_{j,k,j',k'\in C_I:\atop j\ne k, j'\ne k', |\{
j,k,j',k'\}|=2}\cov \bigl[} (1-e_{j'k'})e'_{j'k'}I\bigl(\deg
\bigl(j'\bigr)=d\bigr)I\bigl(\deg\bigl(k'\bigr)\ne
d-1,d\bigr) \bigr]
\\
&&\qquad {}+\frac{c_d}{n^4}\sum_{j,k,j',k'\in C_I:\atop j\ne k, j'\ne k', |\{
j,k,j',k'\}|=3}\cov \bigl[
(1-e_{jk})e'_{jk}I\bigl(\deg(j)=d\bigr)I
\bigl(\deg(k)\ne d-1,d\bigr),
\\
&&\hphantom{\qquad {}+\frac{c_d}{n^4}\sum_{j,k,j',k'\in C_I: \atop j\ne k, j'\ne k', |\{
j,k,j',k'\}|=3}\cov \bigl[} (1-e_{j'k'})e'_{j'k'}I\bigl(\deg
\bigl(j'\bigr)=d\bigr)I\bigl(\deg\bigl(k'\bigr)\ne
d-1,d\bigr) \bigr]
\\
&&\qquad {}+\frac{c_d}{n^4}\sum_{j,k,j',k'\in C_I:\atop |\{j,k,j',k'\}|=4}\cov \bigl[
(1-e_{jk})e'_{jk}I\bigl(\deg(j)=d\bigr)I
\bigl(\deg(k)\ne d-1,d\bigr),
\\
&&\hphantom{\qquad {}+\frac{c_d}{n^4}\sum_{j,k,j',k'\in C_I:\atop |\{j,k,j',k'\}|=4}\cov \bigl[} (1-e_{j'k'})e'_{j'k'}I\bigl(\deg
\bigl(j'\bigr)=d\bigr)I\bigl(\deg\bigl(k'\bigr)\ne
d-1,d\bigr) \bigr].
\end{eqnarray*}
Since $\E e'_{jk}\le c_d/n$, the first two terms in the above bound are
bounded by $c_d/n^3$. To bound the last term, for any $j,k,j',k'\in
C_I$ with $|\{j,k,j',k'\}|=4$, let $C$ be the event that there is no
edge connecting $\{j,k\}$ and $\{j',k'\}$ and define
\begin{eqnarray*}
&&a_{jk}=(1-e_{jk})e_{jk}'I
\bigl(\deg(j)=d\bigr)I\bigl(\deg(k)\ne d-1, d\bigr),
\\
&&\alpha=\E\bigl[(1-e_{jk})e'_{jk}I
\bigl(\deg(j)=d\bigr)I\bigl(\deg(k)\ne d-1,d\bigr)|C\bigr],
\\
&&\beta=\E\bigl[(1-e_{jk})e'_{jk}I\bigl(\deg(j)=d
\bigr)I\bigl(\deg(k)\ne d-1,d\bigr)\bigr].
\end{eqnarray*}
From the conditional independence between $a_{jk}$ and $a_{j'k'}$ given
$C$, $\P(C^c)\le c_d/n$ and $\E e_{jk}'\le c_d/n$, we have
\begin{eqnarray*}
&& \bigl| \cov \bigl[ (1-e_{jk})e'_{jk}I
\bigl(\deg(j)=d\bigr)I\bigl(\deg(k)\ne d-1,d\bigr),
\\
&&\qquad  (1-e_{j'k'})e'_{j'k'}I\bigl(\deg
\bigl(j'\bigr)=d\bigr)I\bigl(\deg\bigl(k'\bigr)\ne
d-1,d\bigr) \bigr] \bigr|
\\
&&\quad = |\E a_{jk}a_{j'k'}-\E a_{jk} \E
a_{j'k'}|
\\
&&\quad =\bigl|\E a_{jk} a_{j'k'}I(C)+\E a_{jk}a_{j'k'}I
\bigl(C^c\bigr)-\beta^2\bigr|
\\
&&\quad =\bigl|\E(a_{jk}|C)\E(a_{j'k'}|C)\P(C)-\beta^2+\E
a_{jk} a_{j'k'}I\bigl(C^c\bigr)\bigr|
\\
&&\quad \le \bigl|\alpha^2-\beta^2\bigr|+\alpha^2 \P
\bigl(C^c\bigr)+\E a_{jk} a_{j'k'}I
\bigl(C^c\bigr)
\\
&&\quad \le 2\E e_{jk}' |\alpha-\beta| +\bigl(\E
e_{jk}'\bigr)^2 \frac{c_d}{n}
\\
&&\quad \le c_d|\alpha-\beta|/n+c_d/n^3.
\end{eqnarray*}
Let
\begin{eqnarray*}
R=(1-e_{jk})I\bigl(\deg(j)=d\bigr)I\bigl(\deg(k)\ne d-1,d\bigr).
\end{eqnarray*}
We have
\begin{eqnarray*}
\alpha-\beta&=&\bigl(\E e_{jk}'\bigr) \bigl(\E(R|C)-\E
RI(C)-\E RI\bigl(C^c\bigr)\bigr)
\\
&=&\bigl(\E e_{jk}'\bigr)\P\bigl(C^c\bigr)
\bigl(\E(R|C)-\E\bigl(R|C^c\bigr)\bigr).
\end{eqnarray*}
Since $\E e_{jk}'\leq c_d/n$ and $\P(C^c)\le c_d/n$, we have $|\alpha
-\beta|\le c_d/n^2$. Therefore,
\begin{eqnarray*}
\Var\bigl(\E\bigl((1-e_{JK})e'_{JK}I
\bigl(\deg(J)=d\bigr)I\bigl(\deg(K)\ne d-1,d\bigr)|V\bigr)\bigr)\le
c_d/n^3.
\end{eqnarray*}
After bounding the variances of the other terms appearing in $\E
(I(V-V'=1)|V)$ by the same argument, we conclude that
%
\begin{equation}
\label{137} \Var\bigl(\E\bigl(I\bigl(V-V'=1\bigr)|V\bigr)\bigr)
\le c_d/n^3.
\end{equation}
Similarly,
%
\begin{equation}
\label{138} \Var\bigl(\E\bigl(I\bigl(V-V'=-1\bigr)|V\bigr)\bigr)
\le c_d/n^3.
\end{equation}
Applying \eqref{136}, \eqref{137} and \eqref{138} in \eqref
{6-l1-1}, we obtain
\eqref{2.5a}, which yields \eqref{2.5}.
\end{pf}

\subsubsection{Uniform multinomial occupancy model}

We consider the uniform multinomial occupancy model studied by Bartroff
and Goldstein \cite{BaGo12}, to which we refer for the literature on
this and related problems. Let $n\ge d\ge2, m\ge2$ be integers. Let
$S$ be the number of urns having occupancy $d$ when $n$ balls are
uniformly distributed among $m$ urns. Bartroff and Goldstein \cite
{BaGo12} proved
\begin{eqnarray*}
d_{K}\bigl(\mathcal{L}(S), N\bigl(\mu,\sigma^2\bigr)
\bigr) \le\frac
{c_d(1+(n/m)^3)}{\sigma},
\end{eqnarray*}
where $\mu, \sigma^2$ are the mean and variance of $S$ given by
%
\begin{eqnarray}
\label{occup-1} &&\mu=m{n \choose d}\frac{1}{m^d}\biggl(1-\frac{1}{m}
\biggr)^{n-d},
\\
\label{occup-2} &&\sigma^2=\mu-\mu^2+m(m-1){n \choose
d,d,n-2d}\frac
{1}{m^{2d}}\biggl(1-\frac{2}{m}\biggr)^{n-2d}
\end{eqnarray}
and $c_d$ is a constant only depending on $d$. Using Corollary~\ref
{6-c3}, we will prove the following bound on the total variation
distance between the distribution of $S$ and $N^d(\mu,\sigma^2)$.
%
\begin{prop}\label{occup-p}
Let $n\ge d\ge2, m\ge2$ be positive integers. Let $S$ be the number
of urns containing $d$ balls when $n$ balls are uniformly distributed
among $m$ urns. Then, with $\mu,\sigma^2$ given by \eqref{occup-1},
\eqref
{occup-2}, we have
%
\begin{equation}
\label{occup-p-1} d_{\mathrm{TV}}\bigl(\mathcal{L}(S), N^d\bigl(\mu,
\sigma^2\bigr)\bigr) \le\frac
{c_d(1+(n/m)^3)}{\sigma},
\end{equation}
where $c_d$ is a constant only depending on $d$.
\end{prop}
%
\begin{rem}
Our approach should also work for the cases $d=0, 1$ if one could prove
similar results as Lemma~3.2 and (3.21) of \cite{BaGo12} and \eqref
{dagger3} below. However, we do not pursue it here.
\end{rem}
\begin{pf}
We follow the construction of size bias coupling in \cite{BaGo12}. For
a given $i\in\{1,\dots, m\}$, we define $m$-dimensional random
vectors $\mathbf{M}_n, \mathbf{M}_n^i$ as follows. Let $\langle\mathbf
{M}\rangle_i$ be the vector obtained by deleting the $i$th component of
$\mathbf{M}$. First, we define the $i$th components of $\mathbf{M}_n,
\mathbf{M}_n^i$ to be $M_n(i)\sim \Binomial(n,1/m), M_n^i(i)=d$. Next,
let $\mathbf{M}_{n,i}', \mathbf{R}_n^i$ be $m$-dimensional random
vectors conditionally independent given $M_n(i)$ such that
$M_{n,i}'(i)=R_n^i(i)=0$ and
\begin{eqnarray*}
\mathcal{L}\bigl(\bigl\langle\mathbf{M}_{n,i}'\bigr\rangle_i|M_n(i)
\bigr)=\operatorname{Multinomial} \bigl(n-\max\bigl\{ M_n(i), d\bigr\}, m-1\bigr)
\end{eqnarray*}
and
%
\begin{equation}
\label{139} \mathcal{L}\bigl(\bigl\langle\mathbf{R}_n^i\bigr\rangle_i|M_n(i)
\bigr)=\operatorname{Multinomial} \bigl(\bigl|d-M_n(i)\bigr|, m-1\bigr),
\end{equation}
where for positive integers $x$ and $y$, $\operatorname{Multinomial}(x,y)$ denotes the
distribution of the numbers of balls in $y$ urns when $x$ balls are
uniformly distributed among them.
Finally, let
\begin{eqnarray*}
\langle\mathbf{M}_n\rangle_i=\bigl\langle\mathbf{M}_{n,i}'\bigr\rangle_i+I
\bigl(M_n(i)<d\bigr)\bigl\langle\mathbf{R}_n^i\bigr\rangle_i
\end{eqnarray*}
and
\begin{eqnarray*}
\bigl\langle\mathbf{M}_n^i\bigr\rangle_i=\bigl\langle\mathbf{M}_{n,i}'\bigr\rangle_i
+ I\bigl(M_n(i)>d\bigr)\bigl\langle\mathbf{R}_n^i\bigr\rangle_i.
\end{eqnarray*}
From the above construction,
\begin{eqnarray*}
\mathcal{L}(\mathbf{M}_n)=\operatorname{Multinomial} (n,m),\qquad  \mathcal{L}\bigl(
\mathbf {M}_n^i\bigr)=\mathcal{L}\bigl(
\mathbf{M}_n|M_n(i)=d\bigr).
\end{eqnarray*}
Therefore, the number of urns having occupancy $d$ in the uniform
multinomial occupancy model can be written as
\begin{eqnarray*}
S=\sum_{j=1}^m I\bigl(M_n(j)=d
\bigr).
\end{eqnarray*}
Define
\begin{eqnarray*}
S^s=\sum_{j=1}^m I
\bigl(M_n^I(j)=d\bigr),
\end{eqnarray*}
where $I$ is uniformly distributed over $\{1,\dots, m\}$ and
independent of all other variables. It was proved in \cite{BaGo12}
that $S^s$ has the $S$-size biased distribution. We are now ready to
apply Corollary~\ref{6-c3}. In the rest of this proof, let $c_d$
denote absolute constants which may depend on $d$, and let $|\cdot|$
denote cardinality when the argument is a set.

By 4(a) of Lemma~3.2 and (3.21) of \cite{BaGo12}, for fixed $d$, there
exists a constant $r_d'$ such that if $\frac{\sigma}{1+(n/m)^3}\geq
r_d'$, then
%
\begin{equation}
\label{occup-4} \sqrt{\Var\bigl(\E\bigl(S^s-S|S\bigr)\bigr)} \le
c_d \frac{1+(n/m)^3}{\sqrt{n}}.
\end{equation}
By (3.18), (3.17), (3.16) and 4(a) of Lemma~3.2 of \cite{BaGo12}, there
exists another constant $r_d''$ such that if $\sigma\geq r_d''$, then
%
\begin{equation}
\label{occup-3} n\le2 m \log m,\qquad  \frac{\mu}{\sigma^2}\le c_d,\qquad
\sigma^2\le c_d n,\qquad  n>\max\bigl\{(d+1)^2, 100
\bigr\}.
\end{equation}
Let $r_d:=r_d'\vee r_d''$.
The range of $n$ and $m$ can be divided into two parts:
\begin{itemize}[(ii)]
\item[(i)] $\frac{\sigma}{1+(n/m)^3}<r_d$,
\item[(ii)] $\frac{\sigma}{1+(n/m)^3}\geq r_d$.
\end{itemize}

Since the total variation distance is always bounded by $1$, \eqref
{occup-p-1} holds true in case (i). Therefore, in the rest of the proof
we only need to consider case (ii), where \eqref{occup-4} and \eqref
{occup-3} hold.

Since $\mathbf{M}_n, \mathbf{M}_n^I$ differ by at most $|M_n(I)-d|+1$
components,
%
\begin{equation}
\label{occup-4a} \bigl|S^s-S\bigr|\le\bigl|M_n(I)-d\bigr|+1.
\end{equation}
Recall that given $I$, $M_n(I)\sim \Binomial(n,1/m)$. From the bounds on
the moments of binomial distributions,
%
\begin{equation}
\label{occup-5} \E\bigl|S^s-S\bigr|^k \le c_d\biggl(1+
\biggl(\frac{n}{m}\biggr)^k\biggr), \qquad k\leq4.
\end{equation}
The first three terms on the right-hand side of \eqref{6-c3-1} are
bounded by $c_d \frac{1+(n/m)^3}{\sigma} $ from \eqref{occup-4},
\eqref
{occup-3} and \eqref{occup-5}. Therefore, to prove Proposition~\ref
{occup-p}, we only need to show that
%
\begin{equation}
\label{occup-6} \E \bigl[ \bigl(\bigl|S^s-S\bigr|^2+\bigl|S^s-S\bigr|
\bigr)d_{\mathrm{TV}}\bigl(\mathcal{L}(S|\mathcal{F}), \mathcal{L}(S+1|
\mathcal{F})\bigr) \bigr]\le c_d \frac{1+(n/m)^3}{\sigma}
\end{equation}
for a $\sigma$-field $\mathcal{F}$ such that $\sigma(S^s-S)\subset
\mathcal{F}$.
Such a $\sigma$-field can be chosen as
\begin{eqnarray*}
\mathcal{F}=\sigma\bigl\{I, M_n(I), \mathbf{R}_n^I,
\bigl\{M_n(j): R_n^I(j)>0\bigr\} \bigr\}
\end{eqnarray*}
from the construction of $\mathbf{M}_n$ and $\mathbf{M}_n^I$.
Write
%
\begin{eqnarray}
\label{occup-7} &&\E \bigl[ \bigl(\bigl|S^s-S\bigr|^2+\bigl|S^s-S\bigr|
\bigr)d_{\mathrm{TV}}\bigl(\mathcal{L}(S|\mathcal{F}), \mathcal{L}(S+1|
\mathcal{F})\bigr) \bigr]\nonumber
\\
&&\quad =\E \biggl[ \bigl(\bigl|S^s-S\bigr|^2+\bigl|S^s-S\bigr|\bigr)\nonumber\\
&&\hphantom{\quad =\E \biggl[} {}\times I\biggl(M_n(I)+\sum_{j:
R_n^I(j)>0}M_n(j)>
\sqrt{n}\biggr) d_{\mathrm{TV}}\bigl(\mathcal{L}(S|\mathcal{F}), \mathcal{L}(S+1|
\mathcal{F})\bigr) \biggr]
\\
&&\qquad {} +\E \biggl[ \bigl(\bigl|S^s-S\bigr|^2+\bigl|S^s-S\bigr|
\bigr)\nonumber\\
&&\hphantom{\qquad {} +\E \biggl[} {}\times I\biggl(M_n(I)+\sum_{j:
R_n^I(j)>0}M_n(j)
\le\sqrt{n}\biggr) d_{\mathrm{TV}}\bigl(\mathcal{L}(S|\mathcal{F}),
\mathcal{L}(S+1|\mathcal{F})\bigr) \biggr].\nonumber
\end{eqnarray}
By the construction of $\mathbf{R}_n^I$, \eqref{139},
\begin{eqnarray*}
\bigl|\bigl\{j: R_n^I(j)>0\bigr\}\bigr|\leq\bigl|d-M_n(I)\bigr|.
\end{eqnarray*}
Also for each $j$ such that $R_n^I(j)>0$,\vspace*{1pt}
%
\begin{equation}
\label{1310} \E\bigl(M_n(j)| I, M_n(I),
\mathbf{R}_n^I\bigr)\leq M_n(I)+\E
B_{n,1/(m-1)}\qquad  \bigl(B_{n,p}\sim \Binomial(n, p)\bigr).
\end{equation}
For the first term on the right-hand side of \eqref{occup-7}, we bound
the total variation distance by $1$, and then apply \eqref{occup-4a},
\eqref
{1310}, and the bounds on the moments of binomial distributions,\vspace*{1pt}
%
\begin{eqnarray}
\label{occup-8} &&\E \biggl[ \bigl(\bigl|S^s-S\bigr|^2+\bigl|S^s-S\bigr|
\bigr) I\biggl(M_n(I)+\sum_{j:
R_n^I(j)>0}M_n(j)>
\sqrt{n}\biggr) d_{\mathrm{TV}}\bigl(\mathcal{L}(S|\mathcal{F}), \mathcal{L}(S+1|
\mathcal{F})\bigr) \biggr]
\nonumber
\\[1pt]
&&\quad \le\frac{2}{\sqrt{n}} \E\bigl(\bigl|M_n(I)-d\bigr|+1\bigr)^2
\biggl( M_n(I)+\sum_{j:
R_n^I(j)>0}M_n(j)
\biggr)
\nonumber
\\[1pt]
&&\quad \le\frac{c_d}{\sqrt{n}} \E \biggl\{ 1+\bigl(M_n(I)
\bigr)^3+\bigl(1+M_n(I)\bigr)^2 \E \biggl(
\sum_{j: R_n^I(j)>0}M_n(j) | I, M_n(I),
\mathbf{R}_n^I \biggr) \biggr\}
\\[1pt]
&&\quad \le\frac{c_d}{\sqrt{n}} \E \bigl\{ 1+\bigl(M_n(I)
\bigr)^3+\bigl(1+M_n(I)\bigr)^3
\bigl(M_n(I)+\E B_{n, 1/(m-1)}\bigr) \bigr\}
\nonumber
\\[1pt]
&&\quad \le c_d \frac{1+(n/m)^4}{\sqrt{n}} .\nonumber
\end{eqnarray}
By observing that for $n\le m$, we have $1/\sqrt{n}\leq c_d/\sigma$
from \eqref{occup-3}, and for $m<n\le2m\log m$, we have (see equation
(3.13) of \cite{BaGo12} with $\varphi_d(n/m)\le1$)\vspace*{1pt}
%
\begin{equation}
\label{occup-8d} \sigma^2 \le c_d m \biggl(
\frac{n}{m}\biggr)^d \mathrm{e}^{-n/m} \qquad \biggl(\mbox{therefore }
\frac{1}{\sqrt{n}}\leq\frac{c_d}{\sigma} \sqrt{\biggl(\frac
{n}{m}
\biggr)^{d-1} \mathrm{e}^{-n/m}} \biggr),
\end{equation}
the bound in \eqref{occup-8} can be further bounded by $c_d/\sigma$.

To bound the second term on the right-hand side of \eqref{occup-7},
for a
given $\mathcal{F}$ with $M_n(I)+\sum_{j: R_n^I(j)>0}M_n(j)\le\sqrt {n}$, let $V$ be the number of urns containing $d$ balls when $n_1$
balls are uniformly distributed among $m_1$ urns where\vspace*{1pt}
%
\begin{eqnarray}
\label{occup-8b} n_1&=&n-\biggl(M_n(I)+\sum
_{j: R_n^I(j)>0}M_n(j)\biggr) \ge n-\sqrt{n}\nonumber
\\[-8pt]\\[-8pt]
&> & d+1 \qquad  \bigl(\mbox{from } n>(d+1)^2 \mbox{ and } d\ge2\bigr)\nonumber
\end{eqnarray}
and\vspace*{1pt}
%
\begin{eqnarray}
\label{1319} m_1&= &m-1-\bigl|\bigl\{j: R_n^I(j)>0
\bigr\}\bigr| \ge m-1-\bigl|M_n(I)-d\bigr|\ge m-1-\sqrt{n}\nonumber
\\[-8pt]\\[-8pt]
&>& 2 \qquad (\mbox{from } n>100 \mbox{ and } n\le2m\log m).\nonumber
\end{eqnarray}
Then $d_{\mathrm{TV}}(\mathcal{L}(V), \mathcal{L}(V+1))=d_{\mathrm{TV}}(\mathcal
{L}(S|\mathcal{F}), \mathcal{L}(S+1|\mathcal{F}) )$.
To apply Lemma~\ref{6-l1}, we construct an exchangeable pair $(V, V')$
by picking a ball uniformly from the $n_1$ balls and distributing it to
an independently and uniformly chosen urn from the $m_1$ urns.
Formally, let $\mathbf{M}_{n_1}$ be an $m_1$-dimensional random vector
with distribution\vspace*{1pt}
\begin{eqnarray*}
\mathcal{L}(\mathbf{M}_{n_1})=\operatorname{Multinomial} (n_1,
m_1).
\end{eqnarray*}
Given $\mathbf{M}_{n_1}$, define two independent random variables $J,
K\in\{1,2,\dots, m_1\}$ with probability mass functions\vspace*{1pt}
\begin{eqnarray*}
\P(J=j)=\frac{M_{n_1}(j)}{n_1},\qquad  \P(K=k)=\frac{1}{m_1}.
\end{eqnarray*}
Given $\mathbf{M}_{n_1}, J, K$, if $J=K$, let $\mathbf
{M}_{n_1}'=\mathbf{M}_{n_1}$, and if $J\ne K$, let $\mathbf
{M}_{n_1}'$ be the $m_1$-dimensional vector with\vspace*{1pt}
\begin{eqnarray*}
M_{n_1}'(J)=M_{n_1}(J)-1,\qquad  M_{n_1}'(K)=M_{n_1}(K)+1
\end{eqnarray*}
and $M_{n_1}'(i)=M_{n_1}(i)$ for $i\ne J, K$. Define\vspace*{1pt}
\begin{eqnarray*}
V=\sum_{j=1}^{m_1} I\bigl(M_{n_1}(j)=d
\bigr)
\end{eqnarray*}
and\vspace*{1pt}
\begin{eqnarray*}
V'=\sum_{j=1}^{m_1} I
\bigl(M_{n_1}'(j)=d\bigr).
\end{eqnarray*}
From the above construction,\vspace*{1pt}
%
\begin{eqnarray}
\label{1311} &&\E\bigl(I\bigl(V-V'=1\bigr)|\mathbf{M}_{n_1}
\bigr)\nonumber\\[1pt]
&&\quad =\sum_{1\le j\ne k \le m_1}\frac
{M_{n_1}(j)}{m_1 n_1} \bigl[ I
\bigl(M_{n_1}(j)=d\bigr)I\bigl(M_{n_1}(k)\ne d-1, d\bigr)
\nonumber\\[1pt]
& &\hphantom{\quad =\sum_{1\le j\ne k \le m_1}\frac
{M_{n_1}(j)}{m_1 n_1} \bigl[}{}+ I\bigl(M_{n_1}(j)\ne d, d+1\bigr) I\bigl(M_{n_1}(k)=d
\bigr) \bigr],
\nonumber\\[-8pt]\\[-8pt]
&&\E\bigl(I\bigl(V-V'=-1\bigr)|\mathbf{M}_{n_1}\bigr)\nonumber\\[1pt]
&&\quad =
\sum_{1\le j\ne k\le m_1}\frac
{M_{n_1}(j)}{m_1 n_1} \bigl[ I
\bigl(M_{n_1}(j)\ne d, d+1\bigr) I\bigl(M_{n_1}(k)=d-1\bigr)\nonumber
\\[1pt]
& &\hphantom{\quad =
\sum_{1\le j\ne k\le m_1}\frac
{M_{n_1}(j)}{m_1 n_1} \bigl[}{} + I\bigl(M_{n_1}(j)=d+1\bigr) I\bigl(M_{n_1}(k)
\ne d-1, d\bigr) \bigr].\nonumber
\end{eqnarray}
Taking expectation on both sides of \eqref{1311},
%
\begin{eqnarray}
\label{1312} &&\P\bigl(V-V'=1\bigr)\nonumber\\
 &&\quad =\sum
_{1\le j\ne k\le m_1} \biggl[ \frac{d}{m_1 n_1} \P \bigl(M_{n_1}(j)=d,
M_{n_1}(k)\ne d-1, d\bigr)
\\
&&\hphantom{\quad =\sum
_{1\le j\ne k\le m_1} \biggl[}{} +\frac{1}{m_1 n_1}\E M_{n_1}(j) I\bigl(M_{n_1}(j)
\ne d, d+1\bigr) I\bigl(M_{n_1}(k)=d\bigr) \biggr].\nonumber
\end{eqnarray}
Let $B_{n,p}\sim \Binomial(n, p)$.
Because binomial distributions do not concentrate on two positive
integers, we claim that for a positive constant $c_d$,
\begin{eqnarray*}
\P(B_{n_1-d, \afrac{1}{m_1-1}}\ne d-1, d)\geq c_d
\end{eqnarray*}
and
\begin{eqnarray*}
\E B_{n_1-d,\afrac{1}{m_1-1}}I(B_{n_1-d,\afrac{1}{m_1-1}}\ne d,d+1 )\geq c_d
\frac{n_1}{m_1}.
\end{eqnarray*}
In fact, recall that $d\geq2$ and write out the binomial probabilities
explicitly, we have
\begin{eqnarray*}
&&\P(B_{n_1-d, \afrac{1}{m_1-1}}=d-1)+\P(B_{n_1-d, \afrac
{1}{m_1-1}}=d)
\\
&&\quad \leq c_d \bigl[ \P(B_{n_1-d, \afrac{1}{m_1-1}}=d-2)+\P(B_{n_1-d,
\afrac{1}{m_1-1}}=d+1)
\bigr],
\end{eqnarray*}
and
\begin{eqnarray*}
&&d\P(B_{n_1-d, \afrac{1}{m_1-1}}=d)+(d+1)\P(B_{n_1-d, \afrac
{1}{m_1-1}}=d+1)
\\
&&\quad \leq c_d \bigl[ (d-1)\P(B_{n_1-d, \afrac{1}{m_1-1}}=d-1)+(d+2)\P
(B_{n_1-d, \afrac{1}{m_1-1}}=d+2) \bigr],
\end{eqnarray*}
which lead to the claim.
Therefore,
%
\begin{eqnarray}
\label{1313} &&\P\bigl(M_{n_1}(j)=d, M_{n_1}(k)\ne d-1, d
\bigr)\nonumber
\\
&&\quad =\P(B_{n_1,\sfrac{1}{m_1}}=d) \P(B_{n_1-d,\afrac{1}{m_1-1}}\ne d-1,d )
\\
&&\quad \geq c_d \P(B_{n_1,\sfrac{1}{m_1}}=d)\nonumber
\end{eqnarray}
and
%
\begin{eqnarray}
\label{1314} &&\E M_{n_1}(j) I\bigl(M_{n_1}(j)\ne d, d+1
\bigr) I\bigl(M_{n_1}(k)=d\bigr)\nonumber
\\
&&\quad =\P(B_{n_1,\sfrac{1}{m_1}}=d) \E B_{n_1-d,\afrac{1}{m_1-1}} I(B_{n_1-d,\afrac{1}{m_1-1}}\ne d,d+1)
\\
&&\quad \geq c_d\frac{n_1}{m_1}\P(B_{n_1,\sfrac{1}{m_1}}=d).\nonumber
\end{eqnarray}
By \eqref{1312}, \eqref{1313} and \eqref{1314},
%
\begin{equation}
\label{occup-8a} \P\bigl(V-V'=1\bigr)\geq c_d\biggl(1+
\frac{m_1}{n_1}\biggr)\P(B_{n_1,\sfrac{1}{m_1}}=d).
\end{equation}
We proceed to bound
$\Var(\E(I(V-V'=1)|V))$ and $\Var(\E(I(V-V'=-1)|V))$. From \eqref
{1311} and the inequality $\Var(X+Y)\le2(\Var(X)+\Var(Y))$, we have
%
\begin{eqnarray}
\label{occup-9} &&\Var\bigl(\E\bigl(I\bigl(V-V'=1\bigr)|V\bigr)\bigr)\nonumber
\\
&&\quad \le\Var\bigl(\E\bigl(I\bigl(V-V'=1\bigr)|\mathbf{M}_{n_1}
\bigr)\bigr)\nonumber
\\[-8pt]\\[-8pt]
&&\quad \le\frac{2}{m_1^2 n_1^2} \biggl[ \Var \biggl(d \sum_{1\le j\ne
k\le m_1}I
\bigl(M_{n_1}(j)=d\bigr)I\bigl(M_{n_1}(k)\ne d-1, d\bigr)
\biggr)\nonumber
\\
&& \hphantom{\quad \le\frac{2}{m_1^2 n_1^2} \biggl[ }{}  + \Var \biggl( \sum_{1\le j\ne k\le m_1}
M_{n_1}(j) I\bigl(M_{n_1}(j)\ne d, d+1\bigr) I
\bigl(M_{n_1}(k)=d\bigr) \biggr) \biggr].\nonumber
\end{eqnarray}
Let
%
\begin{equation}
\label{1315} a_{n_1, m_1}(j,k):=I\bigl(M_{n_1}(j)=d\bigr)I
\bigl(M_{n_1}(k)\ne d-1, d\bigr),
\end{equation}
and let $U_l\in\{1,\dots,m_1\}$ denote the location of the $l$th
ball. Applying the arguments in Bartroff and Goldstein \cite{BaGo12}
(page 17, equation (3.41) and (3.42)),
%
\begin{eqnarray}
\label{1316}
&&\Var \biggl( \sum_{1\le j\ne k\le m_1}a_{n_1-1, m_1}(j,k)
\biggr)\nonumber
\\[-2pt]
&&\quad \le n_1 \E \biggl[ \sum_{1\le k\le m_1, k\ne U_{n_1}} \bigl(
a_{n_1, m_1, (n_1)}(U_{n_1},k)-a_{n_1,m_1}(U_{n_1},k)
\bigr)
\\[-2pt]
&&\hphantom{\quad \le n_1 \E \biggl[}{} + \sum_{1\le j\le m_1, j\ne U_{n_1}} \bigl( a_{n_1, m_1,
(n_1)}(j,
U_{n_1})-a_{n_1, m_1}(j, U_{n_1}) \bigr)
\biggr]^2,\nonumber
\end{eqnarray}
where $a_{n_1, m_1, (n_1)}(j,k)$ is the value of $a_{n_1, m_1}(j,k)$
when withholding ball $n_1$, that is,\vspace*{1pt}
%
\begin{equation}
\label{1317} a_{n_1, m_1, (n_1)}(j,k)=I\bigl(M_{n_1}^{(n_1)}(j)=d
\bigr)I\bigl(M_{n_1}^{(n_1)}(k)\ne d-1, d\bigr)
\end{equation}
with\vspace*{1pt}
\begin{eqnarray*}
M_{n_1}^{(n_1)}(j)= %
\cases{ M_{n_1}(j), &
\mbox{if} $j\ne U_{n_1}$,
\cr
M_{n_1}(j)-1, & \mbox{if}
$j=U_{n_1}$. } %
\end{eqnarray*}
By the definition of $U_l$, given $U_{n_1}$,\vspace*{1pt}
%
\begin{equation}
\label{1318} M_{n_1}(U_{n_1})-1\sim \Binomial
\biggl(n_1-1,\frac{1}{m_1}\biggr).
\end{equation}
Substituting \eqref{1315} and \eqref{1317} in \eqref{1316}, and then applying
the inequality $\E(\sum_{i=1}^n X_i)^2\leq n\E(X_i)^2$ and \eqref{1318},
we have
%
\begin{eqnarray}
\label{star1} &&\Var \biggl( \sum_{1\le j\ne k\le m_1}a_{n_1-1, m_1}(j,k)
\biggr)
\nonumber
\\[-2pt]
&&\quad \le n_1 \E \biggl\{ \sum_{1\le k\le m_1, k\ne U_{n_1}} \bigl[
I\bigl(M_{n_1}(U_{n_1})=d+1\bigr)I\bigl(M_{n_1}(k)
\ne d-1, d\bigr)
\nonumber
\\[-2pt]
&&\hphantom{\quad \le n_1 \E \biggl\{ \sum_{1\le k\le m_1, k\ne U_{n_1}} \bigl[}{}-I\bigl(M_{n_1}(U_{n_1})=d\bigr)I
\bigl(M_{n_1}(k)\ne d-1,d\bigr) \bigr]
\nonumber
\\[-2pt]
&&\hphantom{\quad \le n_1 \E \biggl\{}{} + \sum_{1\le j\le m_1, j\ne U_{n_1}} \bigl[ I
\bigl(M_{n_1}(j)=d\bigr)I\bigl(M_{n_1}(U_{n_1})\ne
d,d+1\bigr)
\nonumber
\\[-2pt]
&&\hphantom{\quad \le n_1 \E \biggl\{{} + \sum_{1\le j\le m_1, j\ne U_{n_1}} \bigl[}{} -I\bigl(M_{n_1}(j)=d\bigr) I\bigl(M_{n_1}(U_{n_1})
\ne d-1, d\bigr) \bigr] \biggr\}^2
\nonumber
\\[-2pt]
&&\quad \le2n_1 m_1 \biggl\{ \sum_{1\le k\le m_1}
\E \bigl[ I\bigl(M_{n_1}(U_{n_1})=d+1\bigr)I
\bigl(M_{n_1}(k)\ne d-1, d\bigr)
\\[-2pt]
&&\hphantom{\quad \le2n_1 m_1 \biggl\{ \sum_{1\le k\le m_1}
\E \bigl[}{}-I\bigl(M_{n_1}(U_{n_1})=d\bigr)I
\bigl(M_{n_1}(k)\ne d-1,d\bigr) \bigr]^2
\nonumber
\\[-2pt]
&&\hphantom{\quad \le2n_1 m_1 \biggl\{}{}+ \sum_{1\le j\le m_1} \E \bigl[ I
\bigl(M_{n_1}(j)=d\bigr)I\bigl(M_{n_1}(U_{n_1})\ne
d,d+1\bigr)
\nonumber
\\[-2pt]
&&\hphantom{\quad \le2n_1 m_1 \biggl\{{}+ \sum_{1\le j\le m_1} \E \bigl[}{} -I\bigl(M_{n_1}(j)=d\bigr) I\bigl(M_{n_1}(U_{n_1})
\ne d-1, d\bigr) \bigr]^2 \biggr\}
\nonumber
\\[-2pt]
&&\quad \le4n_1 m_1 \bigl\{m_1 \bigl[ \P
\bigl(M_{n_1}(U_{n_1})=d+1\bigr) + \P \bigl(M_{n_1}(U_{n_1})=d
\bigr) \bigr] +2m_1 \P\bigl(M_{n_1}(1)=d\bigr) \bigr\}
\nonumber
\\[-2pt]
&&\quad \le c_d n_1 m_1^2 \bigl[
\P(B_{n_1-1, \sfrac{1}{m_1}}=d-1 \mbox{ or } d)+\P(B_{n_1,\sfrac{1}{m_1}}=d) \bigr].\nonumber
\end{eqnarray}
Next, let
\begin{eqnarray*}
b_{n_1, m_1}(j,k):= M_{n_1}(j) I\bigl(M_{n_1}(j)\ne d,
d+1\bigr)I\bigl(M_{n_1}(k)= d\bigr).
\end{eqnarray*}
By the same argument as for $\Var ( \sum_{1\le j\ne k\le
m_1}a_{n_1-1, m_1}(j,k)  )$,
%
\begin{eqnarray}
\label{star2} &&\Var \biggl( \sum_{1\le j\ne k\le m_1}b_{n_1-1, m_1}(j,k)
\biggr)
\nonumber
\\[-2pt]
&&\quad \le n_1 \E \biggl\{ \sum_{1\le k\le m_1, k\ne U_{n_1}} \bigl[
\bigl(M_{n_1}(U_{n_1})-1\bigr)I\bigl(M_{n_1}(U_{n_1})
\ne d+1, d+2\bigr)I\bigl(M_{n_1}(k)=d\bigr)
\nonumber
\\[-2pt]
&&\hphantom{\quad \le n_1 \E \biggl\{ \sum_{1\le k\le m_1, k\ne U_{n_1}} \bigl[}{} - M_{n_1}(U_{n_1})I\bigl(M_{n_1}(U_{n_1})
\ne d, d+1\bigr)I\bigl(M_{n_1}(k)=d\bigr) \bigr]
\nonumber
\\[-2pt]
&&\hphantom{\quad \le n_1 \E \biggl\{ }{} + \sum_{1\le j\le m_1, j\ne U_{n_1}} \bigl[ M_{n_1}(j)I
\bigl(M_{n_1}(j)\ne d, d+1\bigr) I\bigl(M_{n_1}(U_{n_1})=d+1
\bigr)
\nonumber
\\[-2pt]
&&\hphantom{\quad \le n_1 \E \biggl\{ {} + \sum_{1\le j\le m_1, j\ne U_{n_1}} \bigl[} -M_{n_1}(j)I\bigl(M_{n_1}(j)\ne d, d+1\bigr) I
\bigl(M_{n_1}(U_{n_1})=d\bigr) \bigr] \biggr\}^2
\\[-2pt]
&&\quad \le c_d n_1 m_1 \E \biggl\{ \sum
_{1\le k\le m_1, k\ne U_{n_1}} \bigl[\P(B_{n_1-1,\sfrac{1}{m_1}}=d) \E(1+B_{n_1-d-1, \afrac{1}{m_1-1}})^2
\bigr]
\nonumber
\\[-2pt]
&&\hphantom{\quad \le c_d n_1 m_1 \E \biggl\{}{}+ \sum_{1\le j\le m_1, j\ne U_{n_1}} \bigl[\P(B_{n_1-1, \sfrac
{1}{m_1}}=d) \E
B_{n_1-d-1, \afrac{1}{m_1-1}}^2\nonumber\\[-2pt]
&&\hphantom{\quad \le c_d n_1 m_1 \E \biggl\{{}+ \sum_{1\le j\le m_1, j\ne U_{n_1}} \bigl[}{}+\P(B_{n_1-1, \sfrac
{1}{m_1}}=d-1)\E B_{n_1-d, \afrac{1}{m_1-1}}^2
\bigr] \biggr\}
\nonumber
\\[-2pt]
&&\quad \le c_d n_1 m_1^2 \biggl[
\biggl( \frac{n_1}{m_1} +\biggl(\frac
{n_1}{m_1}\biggr)^2
\biggr) \P(B_{n_1-1,\sfrac{1}{m_1}}=d-1) + \biggl( 1 +\biggl(\frac{n_1}{m_1}
\biggr)^2 \biggr) \P(B_{n_1-1,\sfrac{1}{m_1}}=d) \biggr].\nonumber
\end{eqnarray}
From \eqref{occup-9} and the bounds \eqref{star1} and \eqref{star2},
%
\begin{eqnarray}
\label{occup-10} &&\Var\bigl(\E\bigl(I\bigl(V-V'=1\bigr)|V\bigr)\bigr)\nonumber
\\[-2pt]
&&\quad \le\frac{c_d}{n_1} \biggl(1+\biggl(\frac{n_1}{m_1}\biggr)^2
\biggr) \bigl[ \P(B_{n_1, \sfrac
{1}{m_1}}=d-1 \mbox{ or } d)+\P(B_{n_1+1, \sfrac{1}{m_1}}=d)
\bigr]
\\[-2pt]
&&\quad \le\frac{c_d}{n_1} \biggl(1+\biggl(\frac{n_1}{m_1}\biggr)^2
\biggr) \P(B_{n_1, \sfrac
{1}{m_1}}=d-1 \mbox{ or } d).\nonumber
\end{eqnarray}
The last inequality follows from
\begin{eqnarray*}
\P(B_{n_1+1,\sfrac{1}{m_1}}=d)\leq c_d \P(B_{n_1, \sfrac{1}{m_1}}=d)
\end{eqnarray*}
by writing our these probabilities explicitly.

By the same argument as in proving \eqref{occup-9}, \eqref{star1} and
\eqref{star2},
%
\begin{eqnarray}
\label{occup-11} &&\Var\bigl(\E\bigl(I\bigl(V-V'=-1\bigr)|V\bigr)
\bigr)\nonumber
\\
&&\quad \le\frac{2}{m_1^2 n_1^2} \biggl[ \Var \biggl( (d+1) \sum
_{1\le
j\ne k\le m_1}I\bigl(M_{n_1}(j)=d+1\bigr)I
\bigl(M_{n_1}(k)\ne d-1, d\bigr) \biggr)\nonumber
\\
& &\hphantom{\quad \le\frac{2}{m_1^2 n_1^2} \biggl[{}} + \Var \biggl( \sum_{1\le j\ne k\le m_1}
M_{n_1}(j) I\bigl(M_{n_1}(j)\ne d, d+1\bigr) I
\bigl(M_{n_1}(k)=d-1\bigr) \biggr) \biggr]
\\
&&\quad \le\frac{c_d}{n_1} \P(B_{n_1, \sfrac{1}{m_1}}=d \mbox{ or } d+1)\nonumber
\\
&&\qquad {} + \frac{c_d}{n_1} \biggl[ \biggl( \frac{n_1}{m_1} +\biggl(
\frac
{n_1}{m_1}\biggr)^2 \biggr) \P(B_{n_1,\sfrac{1}{m_1}}=d-2) +
\biggl( 1 +\biggl(\frac{n_1}{m_1}\biggr)^2 \biggr)
\P(B_{n_1,\sfrac{1}{m_1}}=d-1) \biggr].\nonumber
\end{eqnarray}
Applying Lemma~\ref{6-l1} with \eqref{occup-8a}, \eqref{occup-10}
and \eqref
{occup-11}, we obtain
\begin{eqnarray*}
&&d_{\mathrm{TV}}\bigl(\mathcal{L}(V), \mathcal{L}(V+1)\bigr)
\\
&&\quad \le\frac{c_d}{\sqrt{n_1}}\frac{1}{(1+\sfrac{m_1}{n_1})\P
(B_{n_1,\sfrac{1}{m_1}}=d) } \\
&&\qquad {}\times\biggl\{ \biggl(\sqrt{
\frac{n_1}{m_1}}+\frac{n_1}{m_1}\biggr)\sqrt{\P (B_{n_1,\sfrac{1}{m_1}}=d-2)}
\\
& &\hphantom{\qquad {}\times\biggl\{}{}+\biggl(1+\frac{n_1}{m_1}\biggr)\sqrt{\P(B_{n_1,\sfrac{1}{m_1}}=d-1
\mbox{ or } d)}\\
& &\hphantom{\qquad {}\times\biggl\{}{} +\sqrt{\P(B_{n_1,\sfrac{1}{m_1}}=d+1) } \biggr\}
\\
&&\quad \le c_d \biggl(1+\sqrt{\frac{n_1}{m_1}}\biggr)
\frac{1}{\sqrt{m_1 (\sfrac
{n_1}{m_1})^d (1-\sfrac{1}{m_1})^{n_1-d} }}.
\end{eqnarray*}
The last inequality was obtained by writing out the binomial
probabilities explicitly. For example,
\begin{eqnarray*}
&&\frac{c_d}{\sqrt{n_1}}\frac{\sqrt{\sfrac{n_1}{m_1}}+\sklfrac
{n_1}{m_1}}{1+\sfrac{m_1}{n_1}} \frac{\sqrt{\P(B_{n_1,\sfrac{1}{m_1}}=d-2)}}{\P(B_{n_1,\sfrac
{1}{m_1}}=d)}
\\
&&\quad =\frac{c_d}{\sqrt{n_1}}\frac{\sqrt{\sfrac{n_1}{m_1}}+\sklfrac
{n_1}{m_1}}{1+\sfrac{m_1}{n_1}} \frac{\sqrt{{n_1 \choose d-2} m_1^{-(d-2)} (1-\sfrac
{1}{m_1})^{n_1-(d-2)} }}{{n_1 \choose d}m_1^{-d}(1-\sfrac
{1}{m_1})^{n_1-d}}
\\
&&\quad \leq\frac{c_d}{\sqrt{n_1}}\frac{\sqrt{\sfrac{n_1}{m_1}}+\sklfrac
{n_1}{m_1}}{1+\sfrac{m_1}{n_1}} \frac{\sqrt{(\sfrac{n_1}{m_1})^{d-2} (1-\sfrac{1}{m_1})^{n_1-d}
}}{(\sfrac{n_1}{m_1})^d(1-\sfrac{1}{m_1})^{n_1-d}}
\\
&&\quad \leq\frac{c_d}{\sqrt{n_1}}\frac{\sqrt{\sfrac{n_1}{m_1}}+\sfrac
{n_1}{m_1}}{1+\sfrac{m_1}{n_1}} \frac{\sfrac{m_1}{n_1}}{\sqrt{(\sfrac{n_1}{m_1})^d(1-\sfrac
{1}{m_1})^{n_1-d}}}
\\
&&\quad \leq c_d\biggl(1+\sqrt{\frac{n_1}{m_1}}\biggr)
\frac{1}{\sqrt{m_1 (\sfrac
{n_1}{m_1})^d (1-\sfrac{1}{m_1})^{n_1-d} }}.
\end{eqnarray*}
From \eqref{occup-8b}, \eqref{1319} and $n\le2m \log m$ in \eqref{occup-3},
we have
%
\begin{equation}
\label{dagger1} n_1\asymp n,\qquad  m_1\asymp m,\qquad
\frac{n_1-d}{m_1^2}\le c_d,
\end{equation}
and hence,
%
\begin{equation}
\label{dagger2} \biggl(1-\frac{1}{m_1}\biggr)^{n_1-d}\ge
\frac{c_d}{(1+\sfrac
{1}{m_1})^{n_1-d}}\ge\frac{c_d}{\mathrm{e}^{(n_1-d)/m_1}}.
\end{equation}
By \eqref{occup-8d}, \eqref{dagger1} and \eqref{dagger2},
%
\begin{eqnarray}
\label{dagger3} &&d_{\mathrm{TV}}\bigl(\mathcal{L}(V), \mathcal{L}(V+1)\bigr)\nonumber\\
&&\quad \le
c_d \biggl(1+\sqrt{\frac
{n_1}{m_1}}\biggr) \frac{1}{ \sigma\sqrt{\mathrm{e}^{n/m} (1-\sfrac
{1}{m_1})^{n_1-d} }}
\\
&&\quad \le\frac{c_d(1+\sqrt{n/m})}{\sigma} \sqrt{\exp\biggl(\frac
{n_1-d}{m_1}-\frac{n}{m}
\biggr)}\le\frac{c_d(1+\sqrt{n/m})}{\sigma}.\nonumber
\end{eqnarray}
This, together with \eqref{occup-5}, proves that the second term on the
right-hand side of \eqref{occup-7} is bounded by
$c_d(1+(n/m)^{5/2})/\sigma$. Therefore, \eqref{occup-6} is proved.
\end{pf}

\section{Proof of Theorem \texorpdfstring{\protect\ref{6-t1}}{1.3}}\label{sec3}

From the definition of $N^d(\mu, \sigma^2)$, (\ref{6.1-0}), we have
%
\begin{equation}
\label{stastasta} d_{\mathrm{TV}}\bigl(\mathcal{L}(S), N^d\bigl(\mu,
\sigma^2\bigr)\bigr)= \sup_{h\in\mathcal{H}} \bigl|\E h(S)-\E
h(Z_{\mu,\sigma^2})\bigr|,
\end{equation}
where $Z_{\mu,\sigma^2}$ is a Gaussian variable with mean $\mu$ and
variance $\sigma^2$ and
%
\begin{equation}
\label{6-t1-3} \mathcal{H}=\biggl\{h:\mathbb{R}\rightarrow\{0,1\}, h(x)=h(z)
\mbox{ when } z-\frac{1}{2} \le x< z+\frac{1}{2} \mbox{ for } z\in
\mathbb{Z}\biggr\}.
\end{equation}
For each $h\in\mathcal{H}$, consider the following Stein equation,
%
\begin{equation}
\label{6-t1-4} \sigma^2 f'(s)-(s-\mu)f(s)=h(s)-\E
h(Z_{\mu,\sigma^2}).
\end{equation}
It is known (see \cite{ChSh05}) that there exists a bounded solution
$f_h$ to (\ref{6-t1-4}) and
%
\begin{equation}
\label{6-t1-5} \Vert f_h\Vert \le\sqrt{\frac{\pi}{2}}
\frac{1}{\sigma}, \qquad \bigl\Vert f_h'\bigr\Vert\le \frac{2}{\sigma^2}.
\end{equation}
By \eqref{stastasta} and \eqref{6-t1-4},
%
\begin{equation}
\label{6-t1-6} d_{\mathrm{TV}}\bigl(\mathcal{L}(S), N^d\bigl(\mu,
\sigma^2\bigr)\bigr)= \sup_{h\in\mathcal{H}} \bigl|\E
\sigma^2 f_h'(S)-\E(S-\mu)
f_h(S)\bigr|.
\end{equation}
Since $(S,S',G)$ satisfies (\ref{6-d1-1}), we have
%
\begin{eqnarray}
\label{6-t1-7} && \E\sigma^2 f_h'(S)-\E(S-
\mu) f_h (S)\nonumber
\\
&&\quad =\E\sigma^2 f_h'(S) -\E\bigl\{ G
f_h\bigl(S'\bigr)-G f_h(S)\bigr\}
\nonumber\\[-8pt]\\[-8pt]
&&\quad = \E\sigma^2 f_h'(S) -\E GD
f_h'(S) -\E G\int_0^D
\bigl(f_h'(S+t)-f_h'(S)
\bigr) \,\mathrm{d}t\nonumber
\\
&&\quad = R_1-R_2,\nonumber
\end{eqnarray}
where
\begin{eqnarray*}
R_1&=&\E f_h'(S) \bigl(
\sigma^2-GD\bigr),
\\
R_2&=&\E G\int_0^D
\bigl(f_h'(S+t)-f_h'(S)
\bigr)\,\mathrm{d}t.
\end{eqnarray*}
From (\ref{6-d1-1}), $\E GD=\sigma^2$. This, along with (\ref
{6-t1-5}), yields
%
\begin{equation}
\label{6-t1-10} |R_1|\le\frac{2\sqrt{\Var(\E(GD|S))}}{\sigma^2}.
\end{equation}
For $R_2$, since $f_h$ solves (\ref{6-t1-4}),
%
\begin{eqnarray}
\label{6-t1-11} R_2&=&\E G\int_0^D
\frac{1}{\sigma^2}\bigl((S+t-\mu) f_h(S+t)-(S-\mu)
f_h(S)+h(S+t)-h(S)\bigr)\,\mathrm{d}t\nonumber
\\[-8pt]\\[-8pt]
&=&\E G\int_0^D \frac{1}{\sigma^2}
\bigl(tf_h(S+t)+(S-\mu ) \bigl(f_h(S+t)-f_h(S)
\bigr)+h(S+t)-h(S)\bigr)\,\mathrm{d}t.\nonumber
\end{eqnarray}
Using (\ref{6-t1-5}), the first two summands in \eqref{6-t1-11} can be
bounded by
\begin{eqnarray*}
\sqrt{\frac{\pi}{8}} \frac{1}{\sigma^3}\E\bigl|GD^2\bigr|+
\frac{1}{\sigma
^4} \E\bigl|GD^2 (S-\mu)\bigr|.
\end{eqnarray*}\eject\noindent
From \eqref{6-t1-3} and \eqref{1.2},
%
\begin{eqnarray}
\label{1101} &&\frac{1}{\sigma^2} \biggl|\E G\int_0^D
\bigl(h(S+t)-h(S)\bigr) \,\mathrm{d}t \biggr|\nonumber
\\
&&\quad = \frac{1}{\sigma^2} \biggl| \E G\int_{-\infty}^\infty
\bigl[I(0\le t\le D)-I(D\le t<0)\bigr] \bigl[\E^{\mathcal{F}} \bigl(h(S+t)-h(S)
\bigr)\bigr]\,\mathrm{d}t \biggr|\nonumber
\\
&&\quad \le \frac{1}{\sigma^2} \E|G|\int_{-\infty}^\infty \bigl|
I(0\le t\le D)-I(D\le t<0) \bigr| \bigl| \E^{\mathcal{F}} \bigl(h(S+t)-h(S)\bigr) \bigr| \,\mathrm{d}t
\\
&&\quad \le \frac{1}{\sigma^2} \E|G|\int_{-\infty}^\infty \bigl|
I(0\le t\le D)-I(D\le t<0) \bigr| \biggl(|t|+\frac{1}{2}\biggr) d_{\mathrm{TV}}
\bigl( \mathcal {L}(S|\mathcal{F}), \mathcal{L}(S+1|\mathcal{F}) \bigr)\, \mathrm{d}t\nonumber
\\
&&\quad \le \frac{1}{2\sigma^2} \E \bigl[ \bigl(\bigl|GD^2\bigr|+|GD|\bigr)
d_{\mathrm{TV}}\bigl( \mathcal {L}(S|\mathcal{F}), \mathcal{L}(S+1|\mathcal{F})
\bigr) \bigr] .\nonumber
\end{eqnarray}
Therefore,
%
\begin{eqnarray}
\label{6-t1-12} |R_2|&\le&\sqrt{\frac{\pi}{8}} \frac{1}{\sigma^3}
\E\bigl|GD^2\bigr|+\frac
{\sqrt{\E G^2 D^4}}{\sigma^3}\nonumber
\\[-8pt]\\[-8pt]
&&{} +\frac{1}{2\sigma^2} \E \bigl[ \bigl(\bigl|GD^2\bigr|+ |GD|\bigr)
d_{\mathrm{TV}} \bigl(\mathcal {L}(S|\mathcal{F}), \mathcal{L}(S+1|\mathcal{F})
\bigr) \bigr].\nonumber
\end{eqnarray}
The theorem is proved by using (\ref{6-t1-6}), (\ref{6-t1-7}) and the
bounds (\ref{6-t1-10}), (\ref{6-t1-12}).

\section*{Acknowledgements}
This work is based on part of the Ph.D. thesis of the author. The
author is thankful to his advisor, Louis H.Y. Chen, for his guidance
and helpful discussions. The author would also like to thank a referee
and the Associate Editor whose suggestions have significantly improved
the presentation of this paper. This work is partially supported by
Grant C-389-000-010-101 and Grant C-389-000-012-101 at the National
University of Singapore.




\printhistory

\end{document}